\documentclass[a4paper,11pt]{article}

\usepackage{amssymb}
\usepackage{amsmath}
\usepackage{amsthm}
\usepackage{array}
\usepackage{authblk}
\usepackage{algorithm}
\usepackage{algorithmicx}%
\usepackage{algpseudocode}
\usepackage{graphicx}
\usepackage[usenames]{color}
\usepackage{multirow}

\usepackage{hyperref}
\newcommand{\R}{\mathbb{R}}
\newcommand{\N}{\mathcal{N}}
\newcommand{\D}{\mathcal{D}}

\newtheorem{Cor}{Corollary}[section]

\theoremstyle{remark}

\newcommand{\cald}{\mathcal{D}_k}

\newcommand{\caln}{\mathcal{N}_k}
\newcommand{\apf}{ f^i_{\mathcal{N}^i_k}}

\makeatletter

\gdef\listctr{list\romannumeral\the\@listdepth}\expandafter
  
\makeatother

\theoremstyle{definition}
\newtheorem{theorem}{Theorem}

\newtheorem{remark}{Remark}

\newtheorem{lemma}{Lemma}
\newtheorem{assumption}{Assumption}
\newtheorem{definition}{Definition}
  {\end{list}}

\title{ASMOP: Additional sampling stochastic trust region  method for multi-objective problems}

\author{Nata\v{s}a Krklec Jerinki\'c \footnote{Department of Mathematics and Informatics, Faculty of Sciences, University of Novi Sad, Trg Dositeja Obradovi\' ca 4, 21000 Novi Sad, Serbia. e-mail: \texttt{natasa.krklec@dmi.uns.ac.rs}},
Luka Rute\v{s}i\'c \footnote{Department of Mathematics and Informatics, Faculty of Sciences, University of Novi Sad, Trg Dositeja Obradovi\' ca 4, 21000 Novi Sad, Serbia. e-mail: \texttt{luka.rutesic@dmi.uns.ac.rs} } \footnote{Corresponding author} 
Ilaria Trombini \footnote{Department of Mathematics and Computer Science, University of Ferrara, Via Machiavelli 30, 44121 Ferrara, Italy.  e-mail: \texttt{ilaria.trombini@unife.it}}
}

\date{}
\begin{document}
\maketitle
\begin{abstract}
 We consider  unconstrained  multi-criteria optimization problems with finite sum objective functions. The proposed algorithm belongs to a non-monotone trust region framework where additional sampling approach is used to govern the sample size and the acceptance of a candidate point. Depending on the problem, the method can yield a mini-batch or an increasing sample size behavior. This work can be viewed as an extension of additional sampling trust region method for scalar finite sum function minimization presented in the literature, requiring nontrivial modifications both in construction and in convergence analysis of the algorithm. Under  assumptions standard for this framework, we  prove stochastic convergence  for twice continuously-differentiable, but possibly non-convex  objective functions. 
 The experiments on machine learning binary classification datasets  show the efficiency of the proposed scheme and its competitiveness with the relevant state-of-the-art methods in both convex and non-convex setup. 
 \end{abstract}

 \section{Introduction}
   The problem we are solving can be stated as \begin{equation}\label{mop}
\min_{x\in \mathbb{R}^n} f(x):=(f^1(x),...,f^q(x))^T
\end{equation}
where $f:R^n\rightarrow R^q$ and each component function is assumed to be smooth with Lipschitz-continuous gradients. Let  $\mathcal{N}^i$, $i=1,...,q$ be  the  respective index sets and  $N^i=|\mathcal{N}^i|=N$, for all $i=1,...,q$.  We assume that each component function has the following finite sum form
\begin{equation} \label{fsfun} f^i(x):=\frac{1}{N}\sum_{j\in \mathcal{N}^i}f^i_j(x), \mbox{ }i=1,...,q.\end{equation}
 This setup is often encountered in Machine Learning (ML) and Deep Learning (DL) and it has been  widely researched, yielding  many crafty optimization methods. The considered problems are often nonlinear, nonconvex, and large-scale, which calls for the development of stochastic optimization methods specialized for  Multi-Objective  (MO) problems. In ML terms, each component function $f^i(x)$ can be seen as a distinct average loss function, where $x\in R^n$ is the vector of trainable parameters for input-label pairs $\{(a_j^i,y_j^i)\}_{j=1}^{N}$ of the training dataset, i.e.,   
$$f^i_j(x):= L^i(a_j^i,y_j^i;x), \quad j=1,..,N, i=1,...,q,$$ 
and $L^i(\cdot)$ measures the prediction error. 
For the  method we propose, the stochastic inflow comes from subsampling which is used to reduce the computational costs of evaluating the full  functions \eqref{fsfun} and possibly the relevant derivatives.

The points of interest in multi-criteria problems are Pareto critical points  which cannot be locally improved in terms of all component function values  \cite{FS},\cite{TOMOP}. Finding Pareto  critical points yields possibility of finding  an  entire Pareto front - a set of globally optimal points \cite{CMV}. This is extremely important in some applications since the representation of the entire front can provide crucial  information. Pareto critical points can be characterized as zeros of the so called  marginal function (see \cite{FS} for more details).  In the single-objective  case  ($q=1$), this concept is reduced to the well known first order optimality conditions, i.e., finding a stationary point of the objective function. Therefore, a common approach in the analysis of multi-criteria optimization is to show that (sub)sequence of marginal function values converges to zero. In  stochastic setup, where only  approximate values of the functions and the derivatives are available,  the approximate marginal function  plays a significant role  in the algorithm, while the analysis aims stochastic convergence of the true marginal function \cite{TE},\cite{NNL}.

Both the line search and the trust region approach has been researched in multi-objective optimization, resulting in a number of  deterministic and stochastic algorithms.   Deterministic multi-criteria steepest descent and Newton method, together with a projected gradient method for the constrained case has been discussed in \cite{SOM}. The stochastic multi-gradient, an extension of the classical stochastic gradient (SG) \cite{ROMO}, can be found in \cite{LV}, in which sublinear
convergence for convex and strongly convex functions is shown.

In \cite{VOS}, the marginal function is utilized to define the trust region method, and therein a convergence towards a critical  point is shown. The complexity of multi-objective problems motivates the development of the stochastic and derivative free approaches. In \cite{TE},  a black box function is used and the derivative is unknown. Therein it is shown that by using the true function values and approximate derivatives it is possible to prove  convergence towards a Pareto optimal point. It is also possible to use both approximate function and gradient values if the estimates are sufficiently accurate with high probability (probabilistically fully linear), see \cite{NNL}, which is a generalization of \cite{CH}. Therein, an adaptive subsampling technique is used which depends on the trust region radius. It successfully reduced the computational cost by using less data when the radius is larger.

The literature also provides methods designed for problems with the finite sum objective functions. These methods exploit the structure of the function  and their advantage lies in subsampling techniques. It is shown that subsampling can help in reducing the costs of deterministic schemes where the full sample set is needed at all iterations, yielding excessive optimization costs. Some papers on this topic are \cite{NKSB},\cite{BBN},\cite{BBN2},\cite{BCN},\cite{BHN},\cite{BCHN},\cite{RM}. Moreover, sampling is inevitable in infinite sum problems that are used for modeling online training process, or in more general problems where the objective function is in a form of mathematical expectation (see \cite{KKSP} or \cite{Ilaria} for instance). In \cite{NKNKJ} an additional sampling technique is employed within a non-monotone trust region framework, aiming to solve single-criterion problems. The idea of non-monotonicity withing trust region framework is also present in the literature \cite{MC},\cite{SN},\cite{CBS}. 

In this work, we propose a stochastic non-monotone trust region algorithm for solving multi-objective problems \eqref{mop}-\eqref{fsfun}. At each iteration we employ  subsampled functions  and gradients to find a candidate subsequent point based on approximate trust region model. The acceptance of that point is based on additional sampling technique \cite{NKNKJ},\cite{NLOS},\cite{IJOT}, \cite{SNKN}, which also governs the subsampling strategy. This means that besides subsampling needed to form an approximate model and the candidate point, we also use an independent subsampled functions as a decision making criteria. Roughly speaking, the role of additional sampling is to test the heterogeneity of the data and increase the level of precision if necessary.  By adaptively choosing the sample size for each component, we  handle the sample average approximation error for each objective function separately.  This leads to two different sample size scenarios: 1)   ``mini-batch scenario", where at least one of the objective functions is approximated  during the whole optimization process; and  
%i.e., the full sample is not reached for at least one objective $f_i$; 
2)  ``full sample scenario", where the full sample is reached eventually for all the objective functions. 

In particular,  the presented work generalizes the additional sampling non-monotone trust region approach  for single-objective problems \cite{NKNKJ} to multi-objective setup. This adaptation requires nontrivial modifications both in the algorithm's construction and in the convergence analysis.  The resulting method yields adaptive sample size sequences designed  to reduce costs by exploiting the problem's structure, while ensuring the convergence towards critical points under some common assumptions for this setup, regardless of the sample size scenario.

The paper is organized as follows. Some basic concepts are covered in the following section. Section 3 presents the proposed algorithm. Within Section 4  the stochastic convergence of the proposed method is analyzed, while Section 5 is devoted to numerical results. Some conclusions are drawn in Section 6. 

\section{Preliminaries}
We start this section by defining efficient and weakly efficient solutions of problem \eqref{mop}. 
\begin{definition}\cite{FS}
    A point $ x^* \in  \mathbb{R}^n $ is called (an) efficient (solution) for \eqref{mop} (or Pareto optimal) if there exists no point $ x \in \mathbb{R}^n$  satisfying $ f^i(x) \leq  f^i(x^*) $ for all $ i \in \{1,2,...,q\} $ and $  f(x) \neq  f(x^*). $ 
A point $ x^* \in \mathbb{R}^n $ is called (a) weakly efficient (solution) for \eqref{mop} (or weakly Pareto optimal) if there exists no point $ x \in  \mathbb{R}^n $ satisfying $ f^i(x) < f^i(x^*) $ for all $i \in  \{1,2,...,q\}.$
\end{definition}
\noindent
It can be shown that a stationarity condition for problem \eqref{mop} is related to marginal function  
\begin{equation}
    \label{marginal}
\omega(x)=-\min_{\|d\|\leq 1}\left(\max_{i \in \{1,...,q\}}\langle \nabla f^i(x),d\rangle\right).
\end{equation}
%The marginal function generalizes the gradient norm in multi-objective settings - 
Notice that $\omega(x)=\|\nabla f(x)\|$ if $q=1$ since the solution of problem \eqref{marginal} is $d^{opt}(x)=-\nabla f(x)/\|\nabla f(x)\|$ in that case. In general, marginal function characterizes Pareto critical points as stated in the following lemma. 
\begin{lemma} \cite{FS}
    \label{lmarginal}
    Let $ {\cal D}(x)$ be the set of solutions of \eqref{marginal}. Then
    \begin{itemize}
        \item[a)]$ \omega(x) \geq 0,$  for every  $  x \in \mathbb{R}^n;$
        \item[b)] if $ x $ is Pareto critical for \eqref{mop} then $ 0 \in {\cal D}(x)$ and $ \omega(x)=0;$
        \item[c)] if $  x $ is not Pareto critical of \eqref{mop} then $\omega(x)>0 $ and any $ d \in {\cal D}(x)$ is a descent direction for \eqref{mop};
        \item[d)] the mapping $ x \to \omega(x)$ is continuous.
    \end{itemize}
\end{lemma}
\noindent 
One possible scalar  representation of the multi-objective problem  \eqref{mop}  is
\begin{equation} \label{scalarrep} \min_{x \in  \mathbb{R}^n}  \phi(x), \; \quad  \phi(x):=\max_{i\in\{1,...,q\}}f^i(x).\end{equation} 
Although problems \eqref{mop} and \eqref{scalarrep} are not equivalent, it can be shown that every  solution  problem \eqref{scalarrep} is a Pareto  critical point of problem \eqref{mop}. 

At each iteration $k$  we form a sample average approximations  of functions \eqref{fsfun} and their derivatives as follows 
\begin{equation}
\label{fcal}
    f^i_{\mathcal{N}^i_k}(x)=\frac{1}{N_k^i}\sum_{j\in\mathcal{N}_k^i}f^i_j(x),\quad  \nabla f^i_{\mathcal{N}^i_k}(x)=\frac{1}{N_k^i}\sum_{j\in\mathcal{N}_k^i}\nabla f^i_j(x),
\end{equation}
where $\mathcal{N}_k^i\subseteq\mathcal{N}^i$ and $N_k^i=|\mathcal{N}_k^i|$. Moreover, we consider the approximate marginal functions \cite{TE}
\begin{equation}
    \label{amarginal}
    \omega_{\caln}(x)=-\min_{\|d\|{\leq}1}\left(\max_{i \in \{1,...,q\}}\langle \nabla\apf(x),d\rangle\right)
    \end{equation}
where $\caln=(\caln^1,...,\caln^q)\subseteq\mathcal{N}=(\mathcal{N}^1,...,\mathcal{N}^q)$ is the set $q$-tuple.  The corresponding scalar problem is then given by 
    \begin{equation} \label{ascalar} 
    \min_{x \in \mathbb{R}^n} \phi_{\caln}(x), \; \quad  \phi_{\caln}(x):=\max_{i\in\{1,...,q\}}\apf(x). 
    \end{equation}
%In deterministic second order trust region framework, the quadratic model of  $\phi(x)$ is given by $$m_k(d):=\max_{i\in\{1,...,q\}} \{ f^i(x_k)+\langle \nabla f^i(x_k),d\rangle +\frac{1}{2}\langle d,H^i_k d\rangle\},$$ where $H^i_k$ is a Hessian approximation of the respective component function. 
Following the ideas from deterministic setup, we form a quadratic model of 
%In general, we will use approximate functions and the gradients as well and thus we relate our quadratic model to 
\eqref{ascalar} as 
\begin{equation}
\label{model} 
m_{\mathcal{N}_k}(d)=\max_{i\in\{1,...,q\}} m_{\caln^i}(d), 
\end{equation}
$$
    m_{\caln^i}(d):=\apf(x_k)+\langle \nabla\apf(x_k),d\rangle+\frac{1}{2}\langle d,H^i_k d\rangle,
$$
where $H^i_k$ approximates the Hessian of function $f^i$ for $i=1,...,q$ at iteration $k$.
Notice that  for each $i=1,...,q$,  $\nabla m_{\caln^i}(0)=\nabla\apf(x_k)$, and $m_{\caln^i}(0)=\apf(x_k)$.

Furthermore, following the approach of  \cite{NKNKJ}, we will rely on additional sampling and draw independently sampled subsets to govern sampling strategy of the method. This yields  a nondecreasing, adaptive  subsampling approach constructed in a such way to avoid employing  the entire sample set for  functions $f^i$ which are homogeneous, i.e., whose subsampled values do not noticeably differ from the full sample values. On the other hand, if the local functions of an objective $f^i$ are heterogeneous, mini-batch does not provide a good representative and the sample size is increased gradually.  As mentioned, this will create  two possible scenarios:  ``mini-batch" (MB) and  ``full sample" (FS). Recall that in  the MB scenario, there exists at least one component function $f^i$ for which the subsampling size $N_k^i$ is strictly less than the full sample size $N$ throughout the algorithm. On the other hand, FS does not imply fully deterministic approach, but rather a non-monotonically  increasing sample mode where the full sample is reached eventually.  More formally, let  us define 
 \begin{equation}
 \label{mb}
     M_b:=\{i\in\{1,...,q\}\; |\; N_k^i<N,\forall k\in \mathbb{N}\}. 
 \end{equation}
 Thus,  $M_b\neq\emptyset$ implies MB scenario, while $M_b=\emptyset$ implies FS scenario. Notice that scenario depends on a sample path - the outcome of the algorithm influenced by random sampling.  Denoting the set of all possible sample paths of the algorithm by $\Omega$, we have  $\Omega=MB\cup FS$ and $MB \cap FS = \emptyset$, where $MB \subseteq \Omega$ represents all possible sample paths such that $M_b\neq \emptyset$, while $FS \subseteq \Omega$ represents all possible sample paths such that $M_b= \emptyset$.

 \section{Algorithm}
Within this section we present  Additional Sampling algorithm  for Multi-Objective Problems - ASMOP. As mentioned earlier, at each iteration $k$, a quadratic model $m_{\caln}(d)$ is formed using the subsampled values \eqref{fcal} and Hessian approximations $H^i_k$. Then, the direction is obtained by solving the following problem approximately
 \begin{equation}
 \label{subp}
     \min_{\|d\|\leq\delta_k}m_{\caln}(d),
 \end{equation} 
 where $\delta_k$ is the trust region radius. The approximate solution $d_k$ is such that  the Cauchy decrease condition holds
 \begin{equation}
 \label{cauchy}
     m_{\caln}(0)-m_{\caln}(d_k)\geq\frac{1}{2}\omega_{\caln}(x_k)\min\{\delta_k,\frac{\omega_{\caln}(x_k)}{\beta_k}\}
 \end{equation}
 with  \begin{equation} \label{betak} \beta_k=1+\max_{i\in\{1,...,q\}}\|H^i_k\|.
 \end{equation} 
 It can be shown that such direction exists (see Lemma 2 of \cite{NNL} for instance). 
 Similar to \cite{NKNKJ}, we will define a trial point $x_t=x_k+d_k$, and check the ratio between the decrease of the relevant scalar function and the quadratic model. Since we are dealing with noisy approximations in general, we adopt non-monotone trust region strategy to avoid imposing a strict decrease   
 and define the ratio as follows
 \begin{equation}
 \label{ronk}
     \rho_{\caln}:=\frac{\phi_{\caln}(x_t)-\phi_{\caln}(x_k)-\delta_kt_k}{m_{\caln}(d_k)-m_{\caln}(0)}
 \end{equation}
 where  $t_k>0$ for all $k$ and 
 \begin{equation}
 \label{tk}
     \sum_{k=0}^{\infty}t_k\leq t<\infty.
 \end{equation}
 The role of  $t_k$ is to control a potential  increase of the scalar function and it is assumed to be  predetermined.
 
Furthermore, let us define the set of indices at iteration $k$ for which we have an approximate function values, i.e.,  
 \begin{equation}
 \label{mbk}
     M^k_b:=\{i\in\{1,...,q\}\; |\; N_k^i<N\}. 
 \end{equation}
 If $M_b=\emptyset$ then we say that  the algorithm is in the FS (full sample) phase at iteration $k$. On the other hand, we say that  the algorithm is in the MB (mini-batch) phase at iteration $k$ if  $M_b\neq \emptyset$. 
If we are in the MB phase, the independent  additional sampling is performed. More precisely, we sample   $\cald=(\cald^1,...,\cald^q)$, independently of $\N_k$, with $\cald^i\subset\mathcal{N}^i$ such that $D_k^i=|\cald^i|<N$ for all $i=1,...,q$. The subsample $\D_k$ is used to  calculate $f^i_{\cald^i}(x_k)$, $f^i_{\cald^i}(x_t)$ and $\nabla f^i_{\cald^i}(x_k)$ by the following formulas 
 \begin{equation}
\label{fdkcal}
    f^i_{\mathcal{D}^i_k}(x)=\frac{1}{D_k^i}\sum_{j\in\mathcal{D}_k^i}f^i_j(x),\mbox{ } \nabla f^i_{\mathcal{D}^i_k}(x)=\frac{1}{D_k^i}\sum_{j\in\mathcal{D}_k^i}\nabla f^i_j(x), \quad i=1,...,q.
\end{equation}
Keep in mind that it is possible that some component functions reached the full sample, while the others did not. 
For each $i\notin M_b^k$ we set $\cald^i=\mathcal{N}^i$, i.e., we have $f_{\D_k^i}(x)=f_{\N_k^i}(x)=f_{\N^i}(x)$.  For all the other components $i\in M_b^k$ it is possible to use subsamples $\D_k^i$ of arbitrary small sizes. In our experiments we set $D_k^i=2$, which keeps the computational cost of additional sampling low.

The following ratio acts as an additional measure of adequacy of the trial point
 \begin{equation} \label{rodk}
     \rho_{\cald}:=\frac{\phi_{\cald}(x_t)-\phi_{\cald}(x_k)-\delta_k\overline{t}_k}{-\max_{i \in \{1,...,q\}}\|\nabla f^i_{\cald^i}(x_k)\|}
 \end{equation}
 where $\phi_{\D_k}(x):=\max_{i\in\{1,...,q\}} f^i_{\D_k}(x)$,  $\overline{t}_k>0$ and 
 \begin{equation}
     \label{tkn}     \sum_{k=0}^{\infty}\overline{t}_k\leq\overline{t}<\infty.
 \end{equation}
 Notice that $\rho_{\cald}\geq\nu$  is equivalent to  an Armijo-like condition
 \begin{equation} \label{ArmDk}
     \phi_{\cald}(x_t)\leq \phi_{\cald}(x_k)+\delta_k\overline{t}_k-\nu\max_{i \in \{1,...,q\}}\|\nabla f^i_{\cald^i}(x_k)\|.
 \end{equation} In the MB phase, the trial point is accepted if both $\rho_{\caln}\geq \eta $ and $\rho_{\cald}\geq \nu$. Notice that the second condition may be very strong since the maximum of the approximate gradients appears on the right-hand side of \eqref{ArmDk}. This is balanced with the choice of predetermined non-monotonicity parameter sequence  $\{\bar{t}_k\}$.  The ratio $\rho_{\cald}$  also plays a significant role  in subsampling strategy as will be seen  within the algorithm. Therefore, the choice  of the non-monotonicity parameter may have a big impact  on   the algorithm's behavior as will be demonstrated within Section 5. If ASMOP reaches the FS phase, only $\rho_{\caln}=\rho_{\mathcal{N}}$ is considered as in the deterministic version of the multi objective trust region \cite{VOS}.

% If the sampling of $\cald$ is done uniformly and randomly, with replacements, then  $f^i_{\cald^i}(x_t)$ represents  a conditionally unbiased estimator of $f^i(x_t)$.  
 
 One more mechanism is used to control the sample size behavior. Namely, if we get relatively close to Pareto critical point on an approximate problem, more precisely, if $\omega_{\caln}(x_k)< \varepsilon h_k^i$ where 
 \begin{equation}
 \label{hk}
     h^i_k:=\frac{N-N_k^i}{N}
 \end{equation} 
 is a sample  average approximation error estimate, we  increase the subsampling size. The algorithm can be stated as follows.

\noindent {\bf{Algorithm 1.}} \textit{(ASMOP)}
\begin{itemize}
\item[Step 0.]  {\it{Initialization.}} \\Choose  $x_0 \in \R^n,\delta_{max}>0, \delta_0\in(0,\delta_{max})$, $\gamma_1\in(0,1),\gamma_2=1/\gamma_1,\nu, \varepsilon >0,\eta\in(0,\frac{3}{4}),$  $\{t_k\}$ satisfying \eqref{tk} and $\{\overline{t}_k\}$ satisfying \eqref{tkn}, $\mathcal{N}_0=(\mathcal{N}_0^1,...,\mathcal{N}_0^q)$.  Set $k=0$.  
\item[Step 1.]  {\it{Candidate point.}} \\Form the model \eqref{model} and find the step  
such that \eqref{cauchy} holds. \\Calculate $\rho_{\caln}$ by \eqref{ronk} and $\omega_{\caln}(x_k)$ by \eqref{amarginal}.\\ Set $x_t=x_k+d_k$.
\item[Step 2.]  {\it{Sample update.}}
\begin{algorithmic}
    \If{$M^k_b=\emptyset$ ({\bf{FS phase)}}}
    \State{Go to Step 3.} 
    \Else{ {\bf{(MB phase)}}} \State{For all $i \in M^k_b$:\\ Choose $\cald^i$ randomly and uniformly, with replacement, from $\mathcal{N}^i$ and calculate     $\rho_{\cald}$ by \eqref{rodk}. }

\If {$\omega_{\caln}(x_k)<\varepsilon h_k^i$} 
    \State{Choose  $N_{k+1}^i \in (N_{k}^i, N]$  and choose $\mathcal{N}_{k+1}^i$.}
\Else
    \If{$\rho_{\cald}<\nu$}
        \State{Choose  $N_{k+1}^i \in (N_{k}^i, N]$  and choose $\mathcal{N}_{k+1}^i$.}
    \Else
        \If{$\rho_{\caln}<\eta$}
            \State{Set $N^i_{k+1}=N^i_k$ and $\mathcal{N}_{k+1}^i=\caln^i$.}
        \Else 
        \State{Set $N^i_{k+1}=N^i_k$, and choose $\mathcal{N}_{k+1}^i$.}
        \EndIf{}
    \EndIf{}
\EndIf{}

 \EndIf{}
\end{algorithmic}

\item[Step 3.] {\it{Iterate update.}} 
\begin{algorithmic}
\If {$M^k_b\neq \emptyset$ {\bf{(MB phase)}}}
    \If{$\rho_{\caln}\geq\eta$ and  $\rho_{\cald}\geq\nu$}
        \State{$x_{k+1}=x_t$.}
    \Else
        \State{$x_{k+1}=x_k$.}
    \EndIf{}    
\Else{ {\bf{(FS phase)}}} 
    \If{$\rho_{\caln}\geq\eta$}
        \State{$x_{k+1}=x_t$.}
    \Else
        \State{$x_{k+1}=x_k$.}
    \EndIf
\EndIf
\end{algorithmic}
\item[Step 4.]  {\it{Radius update.}} 
\begin{algorithmic}
    \If{$\rho_{\caln}\geq\eta$}
        \State{$\delta_{k+1}=\min\{\delta_{max},\gamma_2\delta_k\}$.}
    \Else                   
    \State{$\delta_{k+1}=\gamma_1\delta_k$.}
    \EndIf
\end{algorithmic}
\item[Step 5.]  {\it{Counter update.}}  \\ Set $k=k+1$ and go to Step 1.
\end{itemize}

A few more words are due to algorithm's construction. The initialization includes setting the hyper-parameters, including the trust region radius bound $\delta_{max}$ and the initial sample size $\N_0$. Step 1 finds a candidate point through finding a suitable direction $d_k$ based on the chosen sample $\N_k$. We also calculate the corresponding approximate marginal function. 

Notice that Step 2 is relevant only if the MB phase is active and it is devoted to the sample $\N_k$ update. Besides determining the sample size $N_{k+1}$ to be used in the next iteration, we also determine whether the same sample should be retained. The additional sample $\D_k$ is drawn according to the stated rules to ensure theoretical requirements.  Recall that the sample size of  $D_k^i$ can be arbitrarily small for all $i$ and $k$.
%and that for such small size sets the replacement does not have any influence, since the probability of repeating samples is extremely low for large $N$. However, 
The independence from $\N_k$ is crucial since we need an independent  ``second opinion" on the trial point. 

The sample size can be increased due to $\omega_{\caln}(x_k)<\varepsilon h_k^i$ or $\rho_{\cald}<\nu$. The later case indicates that the trial point obtained via $\N_k$ is a bad choice with respect to  independent representative $\phi_{\D_k}$ of the original scalar representation $\phi$. This signalizes potential heterogeneity of the data and requirement for a better approximation by employing a larger subsample.  If both $\omega_{\caln}(x_k)\geq \varepsilon h_k^i$ and  $\rho_{\cald}\geq \nu$ hold, then we check the agreement $\rho_{\N_k}$. If $\rho_{\N_k}< \eta$ then we suspect that the model $m_{\N_k}$ was not good enough approximation of $\phi_{\N_k}$ and that the failure was due to inadequate radius rather than inadequate sample. In that case, we keep the same sample size, but also retain the same sample  for the next iteration. Finally, if the agreement is good enough, i.e., $\rho_{\N_k}\geq  \eta$, we proceed with the same sample size, but with a different sample. Notice that the sample size sequence is nondecreasing. Therefore, if the algorithm enters the FS phase, it remains in it until the end. 

Step 3 is devoted to the acceptance of the trial point. In the MB phase, both measures of agreement $\rho_{\N_k}$ and  $\rho_{\cald}$ need to be big enough in order to accept the trial point, while in the FS phase the decision is reduced to $\rho_{\N_k}$. The trust region radius is updated within Step 4 where the decision is based solely on $\rho_{\N_k}$ regardless of the phase of the algorithm.

\section{Convergence analysis} 
Within this section we prove almost sure convergence of a subsequence of marginal functions $\omega (x_k)$, which is an equivalent of vanishing subsequence of gradients in the scalar case ($q=1$).  We start the analysis by imposing some  standard assumptions for the considered framework.

\begin{assumption} \label{A1}
All the functions  $f_j^i$, $j\in \mathcal{N}^i$,  $i=1,...,q$ are twice continuously differentiable and bounded from bellow.
\end{assumption}
\noindent Notice that this assumption implies that all the subsampled functions $f^i_{\caln^i}, i=1,...,q$ are twice continuously differentiable and bounded from bellow as well. 
\begin{assumption} \label{A2}
There exists a positive constant $c_{h}$ such that $$\|\nabla^2 f^i_{j}(x)\|\leq c_h \quad \mbox{for all} \quad  x\in \R^n, \; j\in \mathcal{N}^i, \; i=1,...,q,$$ and the sequence of $\beta_k$ defined by  \eqref{betak} is uniformly bounded, i.e., there exists a positive constant $c_{b}$ such that 
$$\beta_k = 1+\max_i\|H_k^i\| \leq c_b  \quad \mbox{for all} \quad k \in \mathbb{N}.$$ 
\end{assumption}
\noindent This assumption implies that all the   Hessians $\|\nabla^2 f^i_{\caln^i}(x)\|$  are uniformly bounded as well with the same constant $c_h$. Although restrictive, this assumption is satisfied in some important classes of ML problems such as logistic regression and linear least squares. 

Similarly to Proposition 5.1 in \cite{VOS} we prove that the error of the approximate model $m_{\caln}$ can be  controlled by the trust region radius. 

\begin{lemma}
\label{fimk}
    Suppose that  Assumptions \ref{A1}-\ref{A2} hold. Then there exists a positive constant $c_f >0$ such that for each $k \in \mathbb{N}$ the step $d_k$ of ASMOP satisfies
    $$|\phi_{\caln}(x_k+d_k)-m_{\caln}(d_k)|\leq c_f\delta_k^2.$$
\end{lemma}
\begin{proof}
Due to Assumption \ref{A1} and Taylor's expansion, for each $i$  there exists $\alpha^k_i\in(0,1)$ such that the following equality holds
\begin{equation} \label{pom1} f^i_{\caln^i}(x_k+d_k)=f^i_{\caln^i}(x_k)+\langle\nabla f^i_{\caln^i}(x_k),d_k\rangle+\frac{1}{2}\langle d_k,\nabla^2f^i_{\caln^i}(x_k+\alpha^k_id_k)d_k\rangle.\end{equation}
Furthermore, there holds
\begin{eqnarray} |\phi_{\caln}(x_k+d_k)-m_{\caln}(d_k)|\nonumber& =& |\max_{i\in\{1,...,q\}}\apf(x_k+d_k)-\max_{i\in\{1,...,q\}} m_{\caln^i}(d_k)|\\ \nonumber &\leq& \max_{i \in \{1,..,q\}}|\apf(x_k+d_k)-m_{\caln^i}(d_k)|\\ 
&\leq & \max_{i \in \{1,..,q\}}(|\frac{1}{2}\langle d_k,\nabla^2f^i_{\caln^i}(x_k+\alpha^k_id_k)d_k\rangle|\\\nonumber&+&|\frac{1}{2}\langle d_k,H_k^i d_k\rangle|)\\\nonumber 
&\leq & \max_{i \in \{1,..,q\}}(\frac{1}{2}\|d_k\|^2 (c_h+c_b)) \leq \frac{1}{2} (c_h+c_b) \delta_k^2 
\end{eqnarray}
and the result holds with $c_f:=\frac{1}{2} (c_h+c_b) $. 
\end{proof}
\begin{remark}
    Notice that if $M_b^k = \emptyset$, i.e., if the algorithm is in the FS phase,  Lemma \ref{fimk} implies  
    \begin{equation}
        |\phi(x_k+d_k)-m_k(d_k)|\leq c_f\delta_k^2.
    \end{equation}
\end{remark}
The further analysis is conducted by observing the two possible outcomes of the algorithm with respect to sample size behavior (MB and FS) separately. However, at the end we combine all the possible outcomes to form the final result stated in Theorem  \ref{glavna}. In order to do that, we follow the similar path as in other papers dealing with the considered additional sampling approach. Let us denote by $\mathcal{D}_k^+$ the subset of all possible outcomes of $\mathcal{D}_k$ such that $\rho_{\cald}\geq\nu$, i.e.,
\begin{equation}
    \mathcal{D}_k^+=\{\mathcal{D}_k\subset\mathcal{N}\;| \; \phi_{\cald}(x_t)\leq \phi_{\cald}(x_k)+\delta_k\overline{t}_k-\nu\max_{i\in \{1,..,q\}}\|\nabla f^i_{\cald^i}(x_k)\|\}.
\end{equation}
If $\mathcal{D}_k\in\mathcal{D}_k^+$ and $\rho_{\caln}\geq\eta$ then $x_{k+1}=x_t$, otherwise we have  
$x_{k+1}=x_k$. Similarly, we denote the set of outcomes where an increase occurs as
\begin{equation}
    \mathcal{D}_k^-=\{\mathcal{D}_k\subset\mathcal{N}\;|\;  \phi_{\cald}(x_t)> \phi_{\cald}(x_k)+\delta_k\overline{t}_k-\nu\max_{i\in \{1,..,q\}}\|\nabla f^i_{\cald^i}(x_k)\|\}.
\end{equation}
Notice that if  $\mathcal{D}_k\in\mathcal{D}_k^-$, then $x_{k+1}=x_k$. 
Now, we proceed  by observing the MB scenario first. 
The following lemma states that, within this scenario, we have  $\rho_{\cald}\geq\nu$ for all $k$ large enough. 
\begin{lemma}
\label{lrodk}
    If $M_b\neq\emptyset$ then  there exists random, finite iteration $k_0\in \mathbb{N}$ such that $\mathcal{D}_k^{-} = \emptyset$ for all $k\geq k_0$. 
\end{lemma}
\begin{proof}
    Since the sample sizes are nondecreasing, for each  $i\in M_b$  there exist a corresponding $\overline{N}^i<N$ and $k_0^i\in \mathbb{N} $ such that $N_k^i=\overline{N}^i$ for all $k\geq k_0^i$. Without loss of generality, let us assume that for all $j \notin M_b$ there holds $N^j_k=N$ for all $k\geq k_0:=\max_{i\in \{1,..,q\}} k_0^i$, i.e.,  for all $k\geq k_0$ the subsample sizes reached their upper limit.
    
    Let us assume the  contrary, that there exist an infinite subsequence of iterations $K\subset \mathbb{N}$ such that $\mathcal{D}_k^{-} \neq  \emptyset$ for all $k\in K$. That means that for all $k\in K$ there exists at least one possible choice of $\cald$ such that $\rho_{\cald}<\nu$.  Without loss of generality, assume that for all $ k\in K$ we have that $k\geq k_0$. 
    Since $D_k^i\leq N-1$, there are finitely many combinations with repetitions for $\cald^i$ and thus there are finitely many choices for $q$-tuples $\cald$\footnote{More precisely, the number of possible choices for $\cald$ is $S(D^i_k)\leq \bar{S}^i:=(2N-2)!/((N-1)!)^2$ where the upper bound follows from the combinatorics of unordered sampling with replacement. Thus, the number of choices for $q$-tuples $\cald$ is also finite and bounded by $\bar{S}=\prod_{i\in M_b} \bar{S}^i.$ }.  
    Therefore there exists  $\tilde{p}\in(0,1)$ such that 
    $P(\cald\in\cald^-)\geq\tilde{p}$, i.e., $P(\cald\in\cald^+)\leq1-\tilde{p}=p<1.$
    Hence 
    $$P(\cald\in\cald^+,\forall k\in K)\leq\prod_{k\in K}p=0,$$
    i.e., almost surely there  exists $k\geq k_0$, such that $\rho_{\cald}<\nu$. However, according to Step 2 of the algorithm,    the sample size is increased for all  $i \in M_b^k$ in that case, which is a contradiction with the assumption that $N_k^i=\overline{N}^i$ for all $i\in M_b$ and $k\geq k_0$. This completes the proof. 
\end{proof}

The following lemma will help us to show that the marginal function tends to zero in the MB scenario. In order to prove the convergence result, we define  an auxiliary function $\Phi_{fix}$  
\begin{equation}
    \label{Fifix}
    \Phi_{fix}(x):=\frac{1}{N} \sum_{j=1}^{N} \max_{i \in \{1,...,q\}} f^i_{j}(x).
\end{equation}
 The result is as follows. 
\begin{lemma}
\label{lfij}
    Suppose that  Assumptions \ref{A1} holds and $M_b\neq\emptyset$. Then 
    $$\Phi_{fix}(x_t)\leq\Phi_{fix}(x_k)-\nu\omega(x_k)+\delta_k\overline{t}_k$$ holds for all $k\geq k_0$, where $k_0$ is as in Lemma \ref{lrodk}.
\end{lemma}
\begin{proof}
   Lemma \ref{lrodk} implies  that we have $\rho_{\cald}\geq\nu,$ i.e.,  
    $$\phi_{\cald}(x_t)\leq \phi_{\cald}(x_k)+\delta_k\overline{t}_k-\nu\max_{i \in \{1,...,q\}}\|\nabla f^i_{\cald^i}(x_k)\|,$$
    for all $k\geq k_0$ and for every possible choice of $\cald$. 
    Since the choice of each  $\cald^i$ is uniform and random with replacements, this further implies\footnote{Additional  sampling is such that it is possible to have $\cald^i=\{j,..,j\}$ which is in fact equivalent of choosing one-element $\cald^i=\{j\}$ due to  sample average approximations.} that the previous inequality also holds for all the single-element choices of $\cald^i$ and all their possible combinations forming $\cald$. 
    Further, observing the combinations of the form  $\cald=(j,...,j)$  for $j=1,...,N$ we obtain 
    $$\max_{i \in \{1,...,q\}} f^i_j(x_t)\leq \max_{i \in \{1,...,q\}} f^i_j(x_k)+\delta_k\overline{t}_k-\nu\max_{i \in \{1,...,q\}}\|\nabla f^i_{j}(x_k)\|.$$
    Now, summing over $j$ and dividing by $N$ we obtain 
    \begin{equation}
        \label{pom5}
    \Phi_{fix} (x_t)\leq \Phi_{fix}(x_k)+\delta_k\overline{t}_k-\nu \frac{1}{N} \sum_{j=1}^{N}\max_{i \in \{1,...,q\}}\|\nabla f^i_{j}(x_k)\|.
    \end{equation}
    Further, let us observe the marginal function and let  us denote by  $d_k^*$  the solution of \eqref{marginal} at iteration $k$, i.e., 
$$\omega(x_k)=-\max_{i\in \{1,..,q\}}\langle\nabla f^i(x_k),d_k^*\rangle.$$
Then for every $i\in \{1,..,q\}$ there holds
$$-\omega(x_k)=\max_{i\in \{1,..,q\}}\langle\nabla f^i(x_k),d_k^*\rangle\geq \langle\nabla f^i(x_k),d_k^*\rangle.$$
Furthermore, by using the Cauchy-Schwartz inequality and the fact that $\|d_k^*\|\leq 1$ we obtain 
$$-\omega(x_k)\geq \langle\nabla f^i(x_k),d_k^*\rangle\geq -\|\nabla f^i(x_k)\|\|d_k^*\|\geq -\|\nabla f^i(x_k)\|.$$
Equivalently, $\omega(x_k)\leq \|\nabla f^i(x_k)\|$ for every $i$ and thus we obtain
\begin{eqnarray}
    \label{pom6} 
    \omega(x_k) &\leq&  \max_{i\in \{1,..,q\}}\|\nabla f^i(x_k)\|=\max_{i\in \{1,..,q\}}\|\frac{1}{N} \sum_{j=1}^{N} \nabla f_j^i(x_k)\|\\\nonumber
    &\leq & \max_{i\in \{1,..,q\}}\frac{1}{N} \sum_{j=1}^{N} \|\nabla f_j^i(x_k)\| \leq \frac{1}{N} \sum_{j=1}^{N} \max_{i\in \{1,..,q\}} \|\nabla f_j^i(x_k)\|.
\end{eqnarray}
Combining this with \eqref{pom5} we obtain the result. 
\end{proof}
The following result states that in the MB case, eventually all the iterates remain within  a random level set of the function $\Phi_{fix}$.  
\begin{Cor}
    \label{lfijk}  Suppose that the assumptions of Lemma \ref{lfij} hold. Then the following holds for all $k \in \mathbb{N}$
     $$ \Phi_{fix}(x_{k_0+k})\leq\Phi_{fix}(x_{k_0})+\delta_{max}\overline{t}.$$
\end{Cor}
\begin{proof}
    Since $x_{k+1}$ is either $x_k$ or $x_t$, Lemma \ref{lfij} and nonnegativity of $\omega(x_k)$ (Lemma \ref{lmarginal}, a)) imply that in either case we have $\Phi_{fix}(x_{k+1})\leq\Phi_{fix}(x_{k}) +\delta_k \bar{t}_k$ for every $k \geq k_0$. Then, the results   follows  from summability of $\bar{t}_k$  \eqref{tkn} and the fact that the sequence of $\delta_k$ is uniformly bounded by $\delta_{max}$. 
\end{proof}
 Similar result can be obtained for the FS scenario as well. 
\begin{lemma}
\label{lfi}
    Suppose that the Assumptions \ref{A1}-\ref{A2} hold. If  $M_b=\emptyset$, then there exists a finite, random iteration $k_1$ such that the following holds for all $k \in \mathbb{N}$
$$\phi(x_{k_1+k})\leq\phi(x_{k_1})+\delta_{max}t.$$
\end{lemma}
\begin{proof}
    Since $M_b=\emptyset$, there exists a random, finite iteration  $k_1\in \mathbb{N}$ such that for all $k\geq k_1$ the full sample is reached, i.e.,  $N_k^i=N$ for all $i=1,...,q$, and thus we have $\phi_{\caln}(x_k)=\phi(x_k)$, $m_{\caln}(d_k)=m_k(d_k)$ and $\omega_{\caln}(x_k)=\omega(x_k)$. Let us observe iterations $k \geq k_1$ and denote $\rho_k:=\rho_{\mathcal{N}}$.
    Then, according to the algorithm, the step is accepted if and only if  $\rho_{k}\geq\eta$. So, either 
    %For all $k\geq k_1$, we know that 
    $\phi(x_{k+1})=\phi(x_k)$ or $\phi(x_{k+1})=\phi(x_{t})$ and 
     $\rho_k\geq\eta$, i.e., 
    \begin{align*}
\phi(x_{k+1})&\leq\phi(x_k)+t_k\delta_k+\eta(m_k(d_k)-m_k(0))\\
        &\leq\phi(x_k)+t_k\delta_k-\frac{\eta}{2}\omega(x_k)\min\{\delta_k,\frac{\omega(x_k)}{\beta_k}\}\\
        &\leq\phi(x_k)+t_k\delta_{max}.
    \end{align*}
    Hence the result follows from \eqref{tk}.
\end{proof}
In order to prove the main result, we assume that the expected value of any local cost function $f^{i}_{j}$ is uniformly bounded. For instance, this is true under the bounded iterates assumption which is common is stochastic framework. More precisely, we assume the following.
\begin{assumption}
\label{A3}
There exists a constant $C>0$ such that for all $i=1,...,q$ and $j\in \mathcal{N}^i$  we have  
$\mathbb{E}(|f^{i}_{j}(x_{k_0})|\; | \;  MB)\leq C$ and $ \mathbb{E}(|f^{i}_{j}(x_{k_1})|\; | \; FS)\leq C.$
\end{assumption}
Notice that Assumption \ref{A3} implies that 
\begin{equation}
    \label{c} 
    \mathbb{E}(\Phi_{fix} (x_{k_0})\; | \;  MB)\leq C \quad \mbox{and} \quad \mathbb{E}(\phi_{fix}(x_{k_1})\; | \; FS)\leq C.
\end{equation}
In the sequel we will show the convergence result for ASMOP algorithm.  The analysis is conducted in a similar way as in the proof of Theorem 1 of \cite{NKNKJ}, but  adapted to fit the  multi-objective framework.  For the sake of readability, we observe separately  MB and FS case. First we consider the MB scenario and show that $\liminf_{k\rightarrow\infty}\omega(x_k)=0$ almost surely.

\begin{theorem} \label{glavnaMB} 
    Suppose that Assumptions \ref{A1}-\ref{A3} hold.  Then 
$$P(\liminf_{k\rightarrow\infty}\omega(x_k)=0 \;| \; MB)=1.$$
\end{theorem}

\begin{proof} Let us consider the MB scenario, i.e., the sample paths such that  $M_b\neq\emptyset$. Then,  Lemma \ref{lrodk} implies  that we have $\rho_{\cald}\geq\nu,$ 
    for all $k\geq k_0$ and for every possible choice of $\cald$. Moreover, we know that for each $i \in M_b$ we have  $N_k^i=\overline{N}^i<N$, and 
    for each $i \notin M_b$ we have  $N_k^i=N$ for all $k\geq k_0$. 
    Moreover, according to Step 2 of the algorithm, we have 
$$\omega_{\caln}(x_k)\geq\varepsilon h_k^i\geq\varepsilon \frac{1}{N}=:\varepsilon_N>0$$
for all $i \in M_b$ and $k\geq k_0$. Now, we will show that there exists an infinite subset of iterations at which the trial point is accepted, i.e., that   $\rho_{\caln}\geq\eta$ occurs infinite number of times. 

Suppose the contrary, that there exists $k_3\geq k_0$ such that   $\rho_{\caln}<\eta$ for all $k\geq k_3$. This further implies that $\lim_{k \to \infty} \delta_k= 0$ due to Step 3 of the algorithm. Moreover, in this scenario, we also have a fixed sample $\mathcal{N}_{k}=\mathcal{N}_{k_3}$ for all $k\geq k_3$. Furthermore, since $m_{\caln}(0)=\phi_{\caln}(x_k)$, from Lemma \ref{fimk} for all $k\geq k_3$ we have that
\begin{align*}
    |\rho_{\caln}-1|&=|\frac{\phi_{\caln}(x_t)-\phi_{\caln}(x_k)-t_k\delta_k}{m_{\caln}(d_k)-m_{\caln}(0)}-1|=|\frac{\phi_{\caln}(x_t)-t_k\delta_k-m_{\caln}(d_k)}{m_{\caln}(d_k)-m_{\caln}(0)}|\\
    &\leq\frac{c_f\delta_k^2+t_k\delta_k}{\frac{\omega_{\caln}(x_k)}{2}\min\{\delta_k,\frac{\omega_{\caln}(x_k)}{\beta_k}\}}\leq\frac{c_f\delta_k^2+t_k\delta_k}{\frac{\varepsilon_N}{2}\min\{\delta_k,\frac{\varepsilon_N}{c_b}\}}.
\end{align*}
Since $\lim_{k \to \infty} \delta_k= 0$, there exists $k_4\geq k_3$ such that  $\delta_k<\varepsilon_N/c_b$ for all $k\geq k_4$ and thus 
$$ |\rho_{\caln}-1|\leq\frac{c_f\delta_k^2+t_k\delta_k}{\frac{\varepsilon_N\delta_k}{2}}=\frac{2c_f\delta_k+2t_k}{\varepsilon_N}.$$
Letting $k \to \infty$ and using the fact that $\lim_{k \to \infty} t_k= 0$ due to \eqref{tk}, we obtain $\lim_{k \to \infty}\rho_{\caln}=1$, which is in contradiction with the assumption of $\rho_{\caln}<\eta<\frac{3}{4}$ for all $k\geq k_3$. 

Thus, we have just shown that there exists an infinite  subsequence $K_2 \subset \mathbb{N}$  such that $\rho_{\caln}\geq \eta$ for all $k\in K_2$. Let $
K_3=K_2\cap\{k_0,k_0+1,...\}:=\{k_j\}_{j \in \mathbb{N}}.$
Then, for all $k\in K_3$ we have  $\rho_{\cald}\geq\nu$ and $\rho_{\caln}\geq\eta$, and thus $x_{k+1}=x_t$. Notice that for all the intermediate iterations, i.e., for all $k\geq k_0$, $k \notin K_3$ we have $x_{k+1}=x_k$. Thus, Lemma \ref{lfij} implies 
$$\Phi_{fix}(x_{k_{j+1}})=...=\Phi_{fix}(x_{k_{j}+1})\leq \Phi_{fix}(x_{k_j})-\nu\omega(x_{k_j})+\delta_{k_j}\overline{t}_{k_j}$$
and thus for each $ j \in \mathbb{N}$
$$\Phi_{fix}(x_{k_{j}})\leq \Phi_{fix}(x_{k_0})-\nu \sum_{l=0}^{j-1}\omega(x_{k_l})+\delta_{max}\overline{t}.$$
Applying the conditional expectation $\mathbb{E}(\cdot\;| \; MB)$ and using Assumptions \ref{A1} which implies that $\Phi_{fix}(x_{k_{j}})$ is bounded from below for any $j$, by employing Assumption  \ref{A3}  and letting $j\to \infty$ we conclude that 
$$\sum_{l=0}^{\infty }\mathbb{E}(\omega(x_{k_l})\;| \; MB)< \infty.$$ 
Now, for any $\epsilon$, from  Markov's inequality we obtain 
$$\sum_{l=0}^{\infty }P(\omega(x_{k_l})\geq \epsilon \;| \; MB)\leq \frac{1}{\epsilon} \sum_{l=0}^{\infty }\mathbb{E}(\omega(x_{k_l})\;| \; MB)< \infty.$$
Finally,   Borel-Cantelli Lemma  implies that $\lim_{ l \to \infty} \omega(x_{k_l})=0$ which completes the proof. 
\end{proof}
Next, we show the same result for the FS scenario. 
\begin{theorem} \label{glavnaFS} 
       Suppose that Assumptions \ref{A1}-\ref{A3} hold.  Then
$$P(\liminf_{k\rightarrow\infty}\omega(x_k)=0 \;| \; FS)=1.$$
\end{theorem}

\begin{proof}
 Let us consider the FS scenario, i.e., the sample paths such that  $M_b=\emptyset$. Then, as in the proof of Lemma \ref{lfi}, for all $k\geq k_1$ we have $N_k^i=N$ for all $i=1,...,q$, and thus $\phi_{\caln}(x_k)=\phi(x_k)$, $m_{\caln}(d_k)=m_k(d_k)$, $\omega_{\caln}(x_k)=\omega(x_k)$ and $\rho_k:=\rho_{\mathcal{N}}$. 

Let us suppose the the statement of this theorem is not true, i.e., that there exists $\varepsilon>0$ and $k_2\geq k_1$ such that $\omega(x_k)>\varepsilon$ for all $k\geq k_2$.  
Since \eqref{tk} implies that $\lim_{k \to \infty} t_k=0, $ without loss of generality, assume that $t_k<c_f$ for all $k\geq k_2$. Then it can be shown that the sequence of $\delta_k$ is uniformly bounded away from zero. Indeed, if at any iteration $k\geq k_2$ the value of $\delta_k$ falls below $\hat{\delta}:=\varepsilon/(20 \max \{1, c_f, c_b\})$, then  Lemma \ref{fimk} implies 
\begin{align*}
    |\rho_k-1|&=|\frac{\phi(x_t)-m_k(d_k)-t_k\delta_k}{m_k(d_k)-m_k(0)}|\leq\frac{c_f\delta_k^2+t_k\delta_k}{\frac{\omega(x_k)}{2}\min\{\delta_k,\frac{\omega(x_k)}{\beta_k}\}}\\
&\leq\frac{c_f\delta_k^2+t_k\delta_k}{\frac{\varepsilon}{2}\min\{\delta_k,\frac{\varepsilon}{c_b}\}}\leq
\frac{c_f\delta_k^2+t_k\delta_k}{\frac{\varepsilon}{2}\delta_k}<
\frac{c_f\hat{\delta}+c_f\hat{\delta} }{\frac{\varepsilon}{2}}<\frac{1}{4}
\end{align*}
which further implies that  $\rho_k>\frac{3}{4}>\eta$, and, according to the algorithm, the radius is increased, i.e., $\delta_{k+1}>\delta_k$. Therefore, there exists $\tilde{\delta}>0$, such that $\delta_k>\tilde{\delta}$, for all $k\geq k_2$.

On the other hand, the existence of $k_3$ such that $\rho_k < \eta $ for every $k \geq k_3$ would imply $\lim_{k\rightarrow\infty}\delta_k=0$ due to Step 4 of the algorithm, which is in contradiction with $\delta_k>\tilde{\delta}$, for all $k\geq k_2$. Thus, we conclude that there must exist an infinite number of iterations  $K_1\subset \mathbb{N}$ such that for all $k \in K_1$ there holds $k\geq k_2$ and $\rho_k\geq\eta$. 
Therefore, for all $k\in K_1$ we have 
\begin{align*}
\phi(x_{k+1})&\leq\phi(x_k)+\delta_kt_k+\eta(m_k(d_k)-m_k(0))\\
    &\leq\phi(x_k)+\delta_k t_k-\eta\frac{\omega(x_k)}{2}\min\{\delta_k,\frac{\omega(x_k)}{\beta_k}\}\\
    &\leq\phi(x_k)+\delta_k t_k-\eta\frac{\varepsilon}{2}\min\{\tilde{\delta},\frac{\varepsilon}{c_b}\}\\
    &=:\phi(x_k)+\delta_k t_k-\hat{c}.
    %\frac{\hat{c}}{2}
\end{align*}
Again, without loss of generality, we can assume that for all $k \in K_1$ there holds $\delta_k t_k\leq \hat{c}/2$ and thus 
$$\phi(x_{k+1})\leq \phi(x_k)-\frac{\hat{c}}{2}, \quad k \in K_1.$$
Furthermore, by denoting $K_1=\{k_j\}_{j \in \mathbb{N}}$ and using the fact that for all $k\geq k_2$ such that $k \notin K_1$ the trial point is rejected and thus $\phi(x_{k+1})=\phi(x_k)$, we conclude that for all $j \in \mathbb{N}$ 
$$\phi(x_{k_{j+1}})=...=\phi(x_{k_{j}+1})\leq\phi(x_{k_j})-\frac{\hat{c}}{2}. $$
Applying the conditional expectation and using Assumption \ref{A3} we obtain that for every $j \in \mathbb{N}$
$$\mathbb{E}(\phi(x_{k_{j}})\; | \; FS)\leq C-j \frac{\hat{c}}{2} $$
and letting $j\to \infty$ we obtain 
$\lim_{j\rightarrow\infty}\mathbb{E}(\phi(x_{k_j})\;|\;FS)=-\infty$. This is in contradiction with Assumption \ref{A1} which implies that the function $\phi$ is bounded from below. This completes the proof. 
\end{proof}

Finally, we obtain the following result as a direct consequence of the previous two theorems and Lemma \ref{lmarginal}. 

\begin{Cor}
    \label{glavna}  Suppose that Assumptions \ref{A1}-\ref{A3} hold.  Then the sequence of iterates $\{x_k\}_{k \in \mathbb{N}}$ generated by the ASMOP algorithm satisfies   
$$\liminf_{k\rightarrow\infty}\omega(x_k)=0 \quad \mbox{a.s.}$$
Moreover, if the sequence of iterates is bounded, then a.s. there exists an accumulation point of the sequence $\{x_k\}_{k \in \mathbb{N}}$ which is a Pareto critical point for problem \eqref{mop}. 
\end{Cor}

\section{Numerical results}
In this section we showcase numerical results obtained on machine learning problems. The proposed algorithm - ASMOP is compared to the  state-of-the-art stochastic multi-gradient method  SMG \cite{LV} for  multi-objective problems.   We also compare ASMOP to  SMOP \cite{NNL} which also uses sample average  approximations, but different sample size guidance.  The comparison is made through FEV (number of function evaluations) and CPU time.  The first measure is relevant for problems where accessing the data represents a dominant cost and  it also serves as a proxy for computational cost. The other measure is relevant in applications such as portfolio optimization in finance, where the approximate solution needs to be obtained promptly. It is, however, sensitive to  implementation issues which may not be optimal for all the considered methods. Nevertheless, the results reveal potentials of ASMOP considering both measures. 

We tested the relevant methods on different types of problems.  Section 5.1 covers convex problems where the objectives are constructed through regularized logistic regression models. In Section 5.2 we present the results obtained on different types of objectives - one objective is modeled by regularized nonconvex loss in 2-layer Neural Networks, while the other one represents a least squares problem coming from, e.g.,  linear regression.    Additionally, within Section 5.3  we  compare different parameter configurations of ASMOP and demonstrate algorithm's behavior under different sample size increments.
In all sections, since the gradients and Hessians reuse intermediate quantities computed during the loss evaluation (e.g., $x^Ta_j$), they can be obtained with a negligible additional computational cost. Therefore, one FEV   accounts for the joint evaluation of the function, gradient, and Hessian of a single function $f^i_j$.
\subsection{Logistic regression}
%In our experiments we created multi-objective problems using different methodologies. 
The first set of experiments consists of minimizing the regularized logistic regression loss function with the CIFAR10, MNIST and Fashion MNIST dataset. CIFAR10 is a dataset for an image classification problem, and it consists of $N=5\times10^4$ RGB images of $32\times32$ pixels in 3 hues, which are in 10 categories, hence the dimension of the problem is  $n=32\times32\times3=3072$. We create a model which differentiates category 0 (airplane) from 1 (car), and also category 2 (bird) from 3 (cat) at the same time.  Assuming $\mathcal{N}$ is the set of indices of the training dataset, we split the dataset into two subsets by creating two index subgroups $\mathcal{N}^1$ and $\mathcal{N}^2$ such that $\mathcal{N}^1$ are indices of the samples which are in categories $0$ and $1$, and $\mathcal{N}^2$ are indices of the samples in categories $2$ and $3$.  The second dataset we used was the MNIST dataset which consists of $N=7\times10^4$ samples of gray-scale images of handwritten digits with $32\times32$ pixels, hence the dimension is $n=1024$. Once again, we created two datasets from MNIST, the first containing samples with labels $0$ and $8$, and the second with labels $1$ and $4$. This way the model differentiates digits $0$ or $8$, and  digits $1$ and $4$ at the same time, similarly as with CIFAR10. We made both subgroups  contain the same amount of elements $N=10^4$. 
The third dataset we used was the Fashion MNIST dataset which consists of $N=7\times10^4$ samples of gray-scale images of clothing items with $32\times32$ pixels, hence the dimension is $n=1024$. Once again, we created two datasets from Fashion MNIST, the first containing samples with labels $0$ and $1$, corresponding to T-shirt/top and trouser, and the second with labels $2$ and $3$, corresponding to pullover and dress. This way the model differentiates T-shirt/top or trouser, and pullover or dress at the same time, similarly as with CIFAR10 and MNIST. We made both subgroups contain the same amount of elements $N=10^4$.
Finally, we considered the MNIST-Fairness dataset, which is derived from the MNIST dataset. As before, the original dataset consists of $N=7\times10^4$ gray-scale images of handwritten digits with $32\times32$ pixels, hence the dimension is $n=1024$. From this dataset we selected only the samples with labels $1$ and $4$, corresponding to the handwritten digits $1$ and $4$.

In order to introduce a sensitive attribute, we considered the pixel at position $162$ and split the samples according to whether the pixel value was smaller or larger than its average value over the dataset. This procedure generates two subgroups of data, which can be interpreted as two sensitive groups. 
As mentioned, we are minimizing a regularized logistic regression loss function
\begin{equation}
\label{logreg}
    \min_{x\in \mathbb{R}^n}f(x):= (f^1(x),f^2(x)),
\end{equation}
where the function components are defined as follows:
\begin{equation}
\label{logregf}
f^i(x)=\frac{1}{N}\sum_{j\in \mathcal{N}^i}\log(1+e^{(-y_j(x^Ta_j))})+\frac{\lambda_i}{2}\|\hat{x}\|^2,i=1,2.    
\end{equation}
Here,  
$x\in \mathbb{R}^n$ represents model coefficients we are trying to find, $\hat{x}$ coefficient vector without the intercept,  $a_j$ the feature vector 
and $y_j$ the relevant label.

There are several factors and parameter choices which influence the behavior of the algorithm. For this experiment, we set the initial subsampling size $N_0^i=0.01N$, and we update the size at Step 2 when needed by a small amount, $\Delta N_k^i=0.02N$. For the nonmonotonicity parameters we set $t_k=\frac{1}{(k+1)^{1.51}}$, and $\overline{t}_k=\frac{100}{(k+1)^{1.51}}$, which satisfy  \eqref{tk} and \eqref{tkn}. Since the power of the denominator is slightly larger than 1, the algorithm will slowly and gradually decrease the tolerance towards the rise of the function value. Both additional sampling sizes $D_k^1$ and $D_k^2$ are set to $2$  for all iterations in order to keep the additional sampling computationally cheap. We compare algorithms' true marginal function values $\omega(x_k)$ as a measure of stationarity for multi-objective problems. We calculate $\omega(x_k)$ at each iteration and plot it against the  FEV or CPU time. 
For all the considered methods, we perform $5$ runs with the same parameters and report the plots of the averaged values. In the plots of $\omega(x_k)$, a shaded region is shown to represent the Standard Deviation.

The following Figures \ref{cifar}, \ref{mnistfig}, \ref{mnistfashionfig} and \ref{mnistfairnessfig} compare the three mentioned algorithms for the fixed budget of $8\cdot10^6$ function evaluation and 80 seconds for CIFAR10 dataset, $10^6$ and 3 seconds for MNIST dataset, $2\cdot10^6$ and 6 seconds for Fashion MNIST dataset and $10^5$ and $0.3$ seconds for MNIST-Fairness dataset. We have also included the subsampling sizes of SMOP and ASMOP for both function components in all figures. It is evident that ASMOP uses small amount of information to achieve large function reduction. We have noticed that if we set $1\%$ of samples as a starting size, and increase the size of both subsamples by $2\%$ of the respective maximum sample size, the subsampling sizes update similarly for both criteria $f^1$ and $f^2$. 
\begin{figure}[H]
\centering
\includegraphics[width=5.7cm]{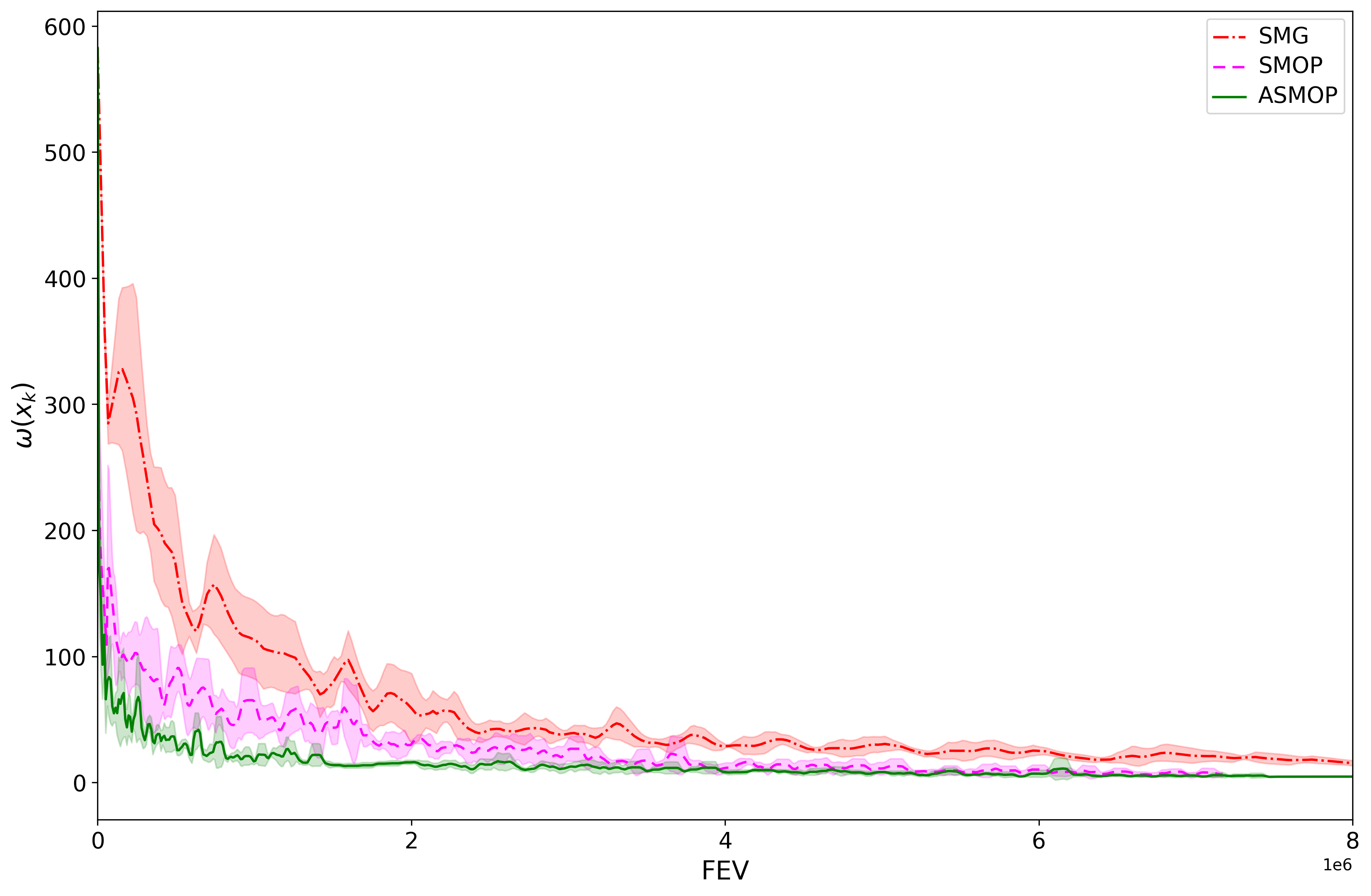}\quad
\includegraphics[width=5.7cm]{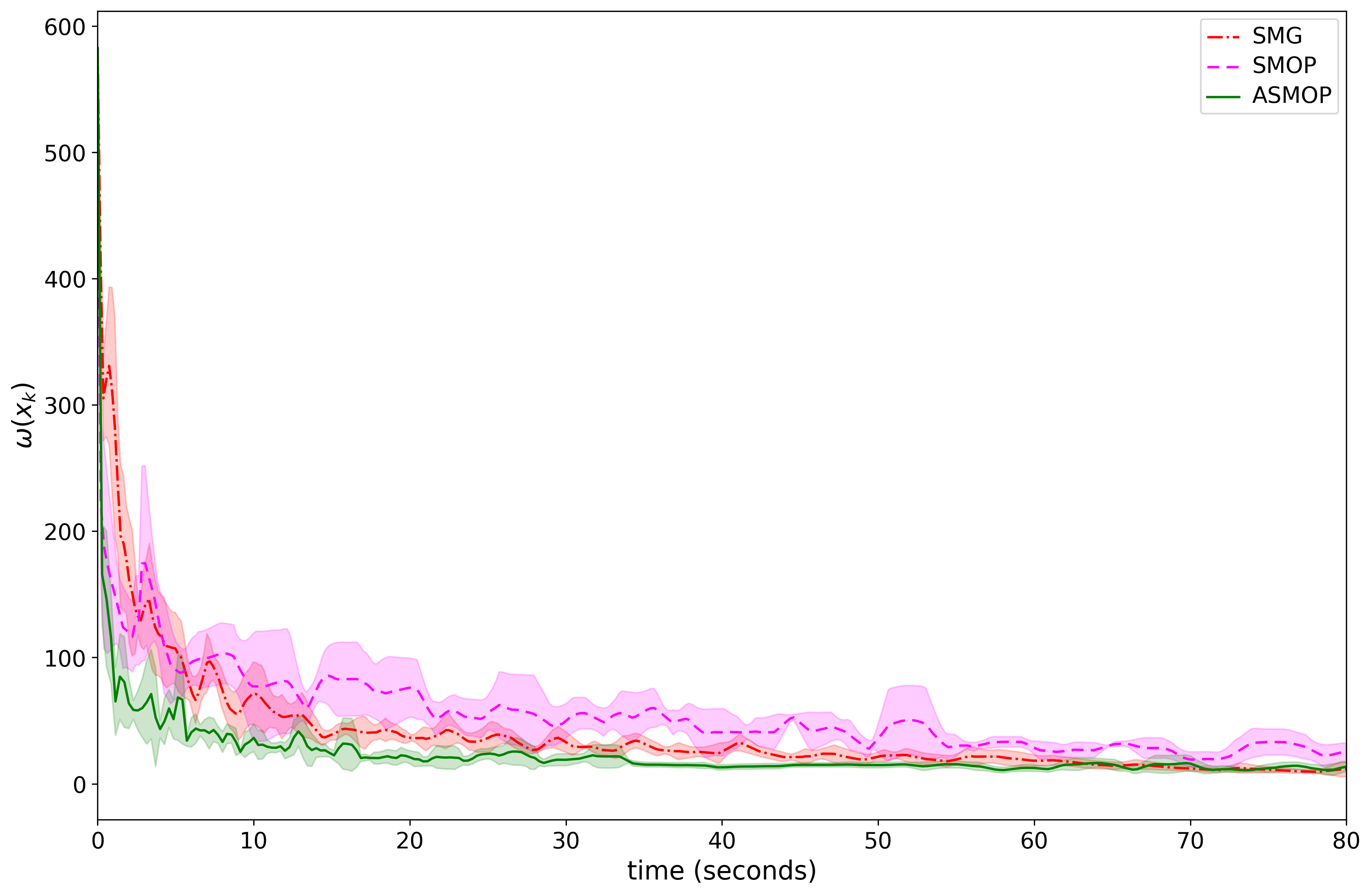}\\
\includegraphics[width=6.27cm]{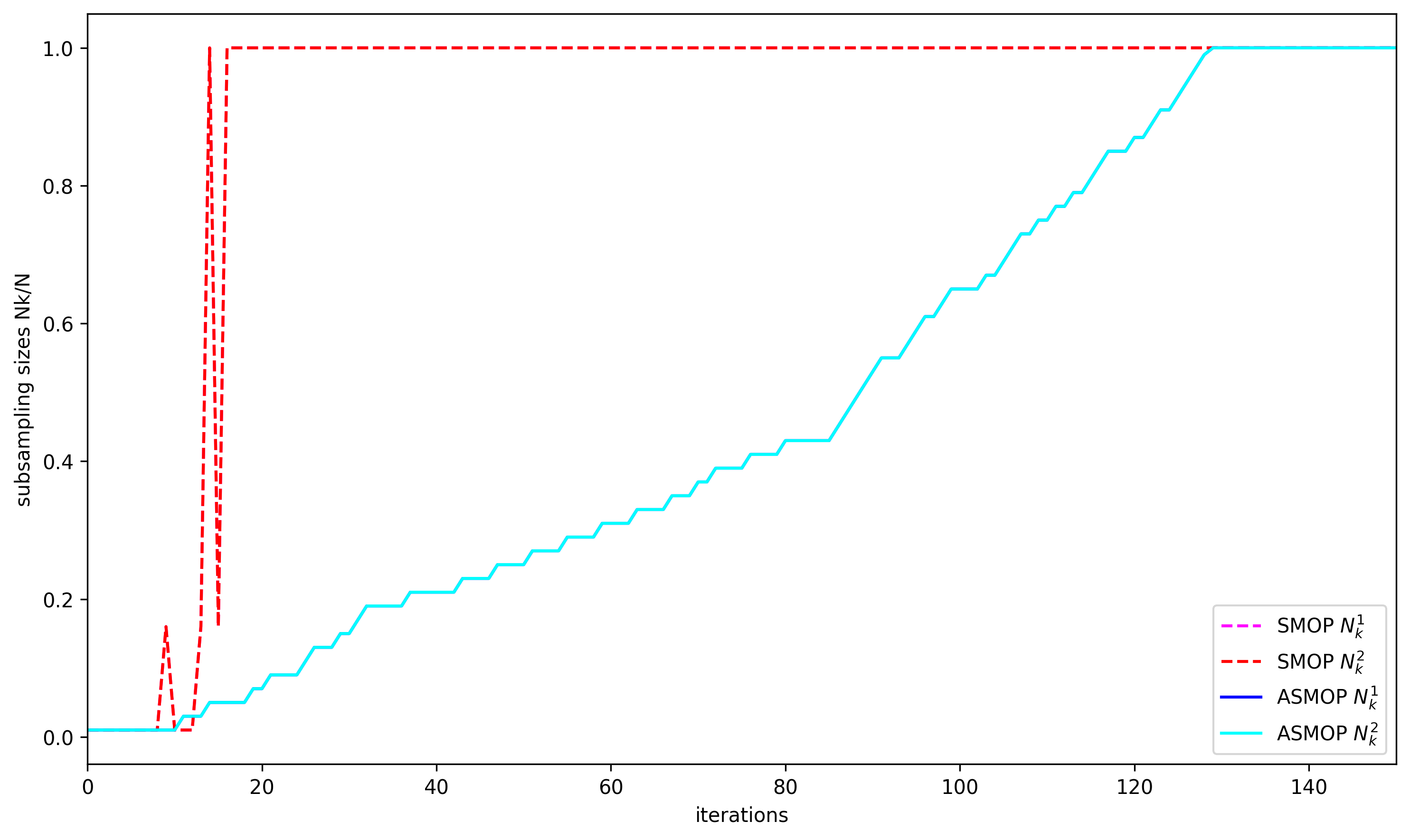}
\caption{{\footnotesize{CIFAR10 dataset, problem \eqref{logregf}, $N=10^4,n=3072$. First row: optimality measure against function evaluations (left) and optimality measure against time in second (left). Second row: sample sizes behavior. Parameters: $x_0=(0.1,0.1,...,0.1), \delta_0=1, \delta_{max}=8, \gamma_1=0.5, \gamma_2=2, \nu=10^{-4}, \eta=0.25,\varepsilon=10^{-9}.$}}}	
\label{cifar}
\end{figure}

\begin{figure}[H]
\centering
\includegraphics[width=5.7cm]{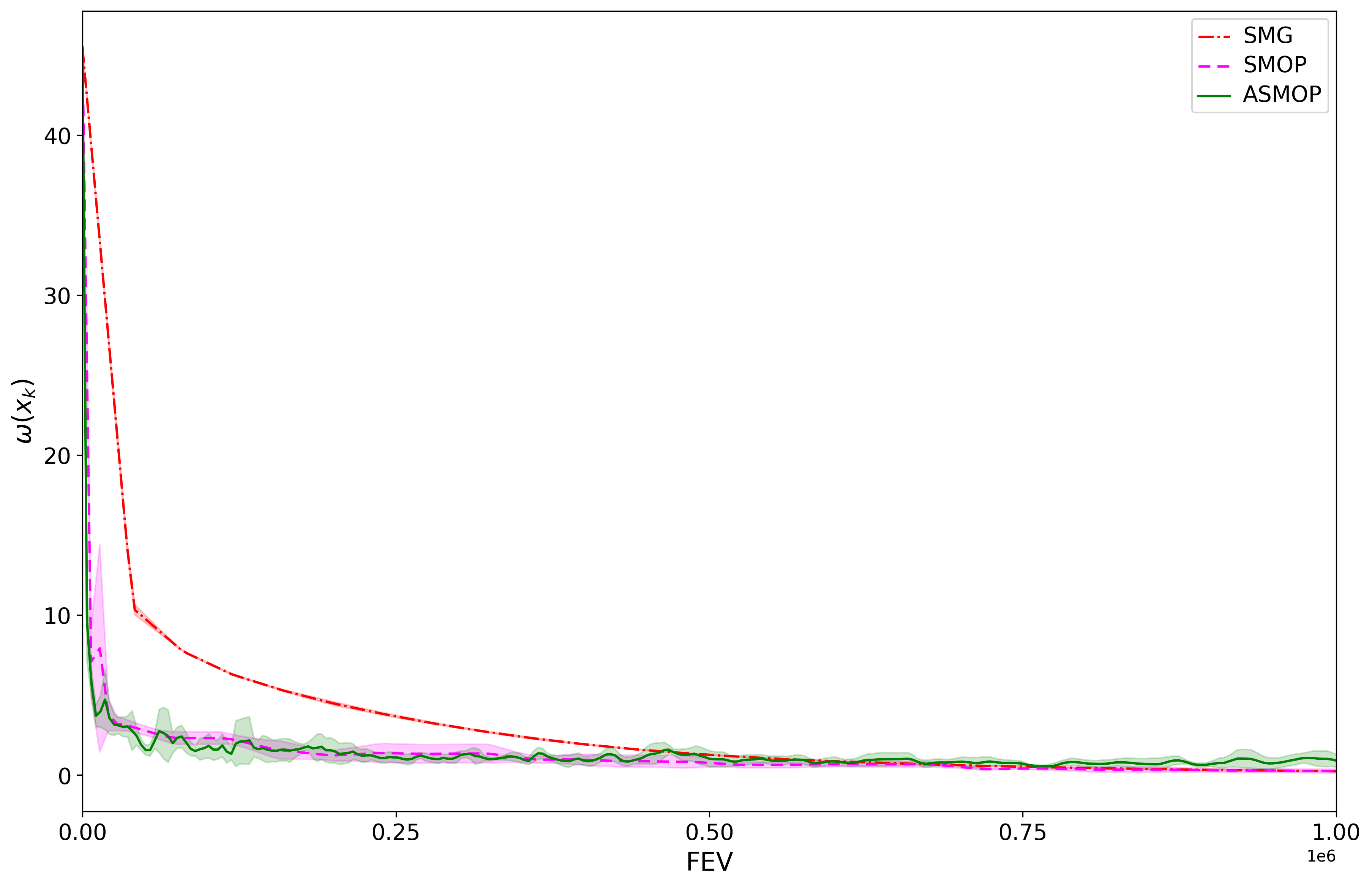}\quad
\includegraphics[width=5.7cm]{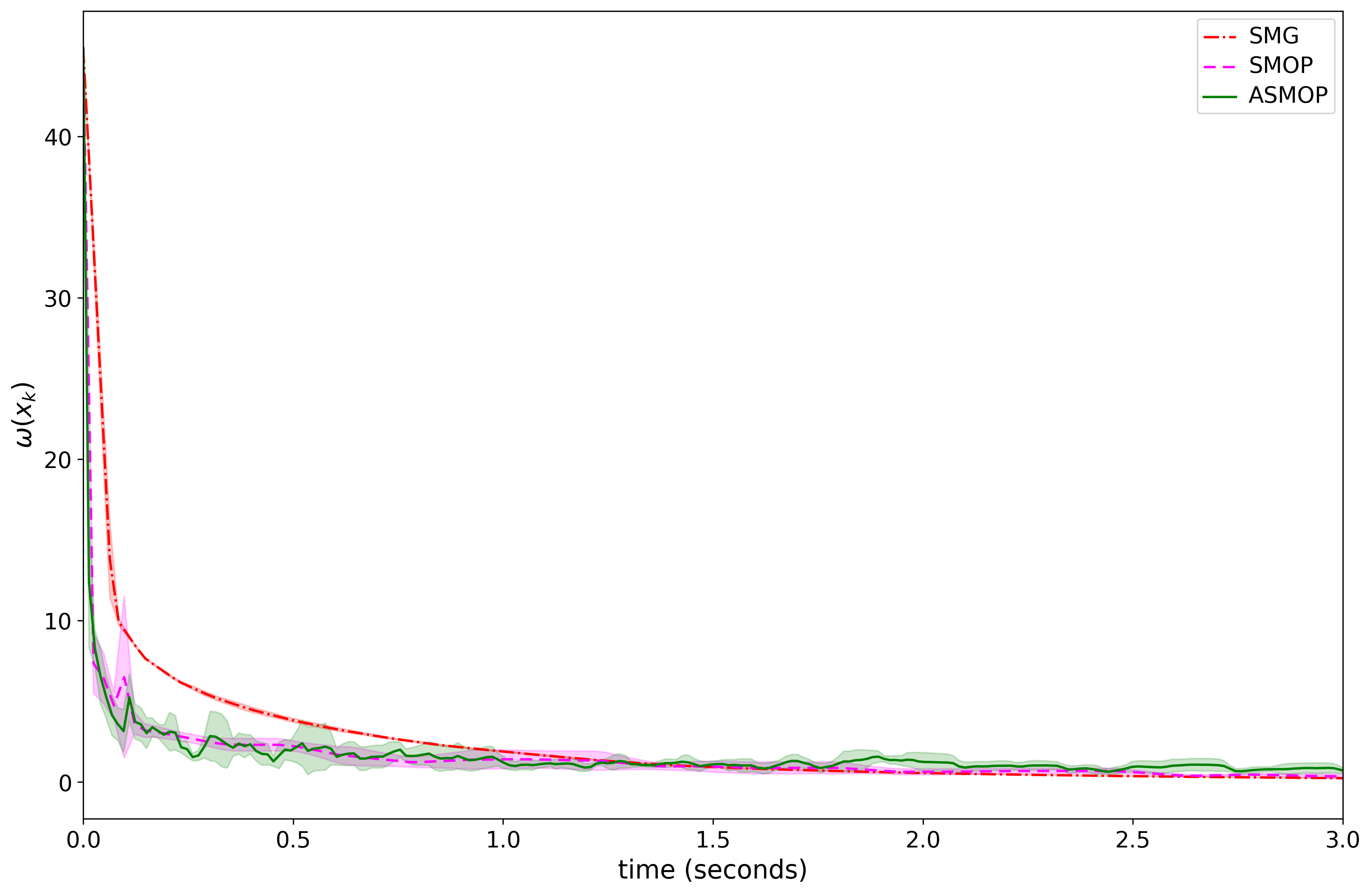}\\
\includegraphics[width=6.27cm]{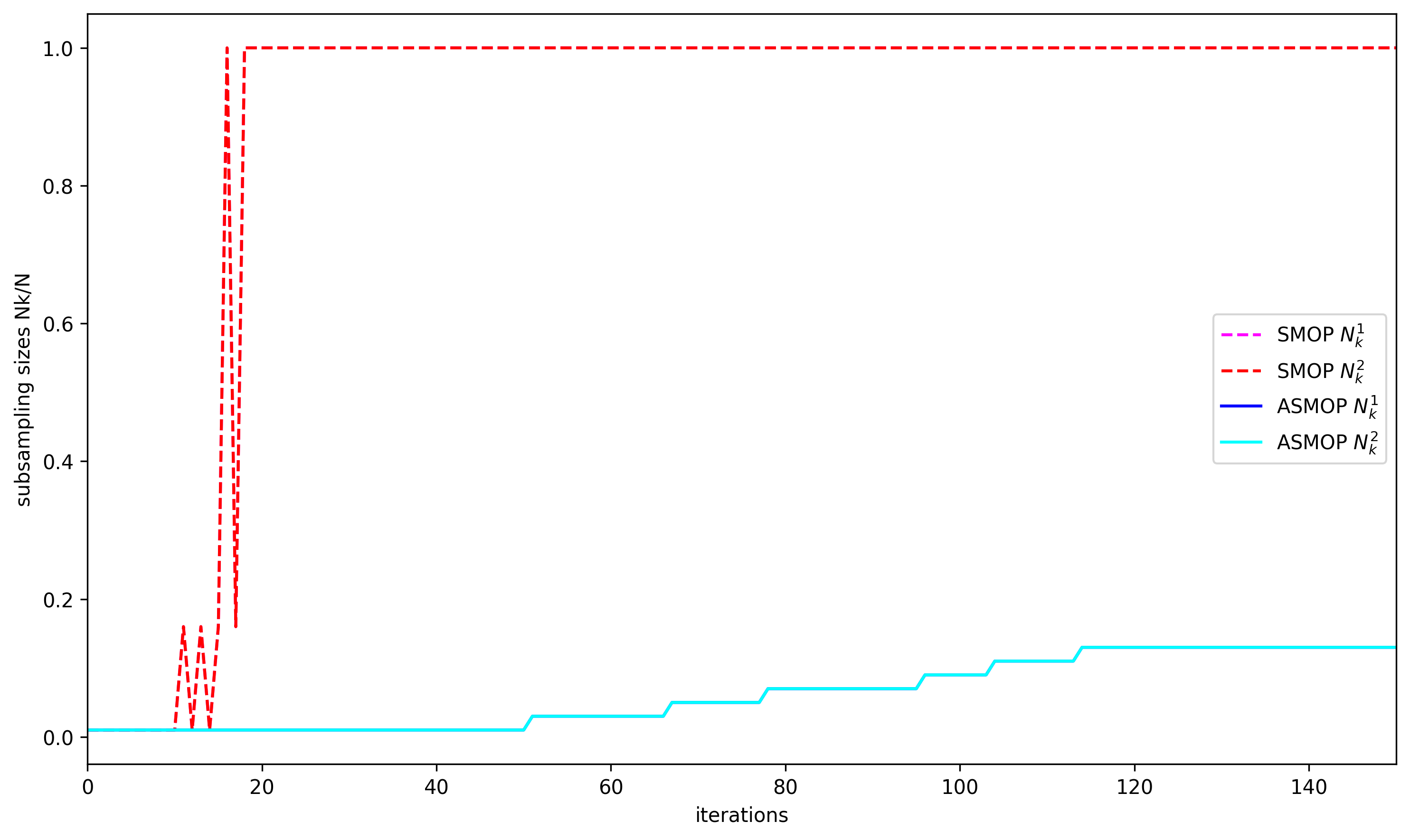}

\caption{{\footnotesize{MNIST dataset, problem \eqref{logregf},  $N=10^4,n=1024$. First row: optimality measure against function evaluations (left) and optimality measure against time in second (left). Second row: sample sizes behavior. Parameters: $x_0=(0.1,0.1,...,0.1), \delta_0=1, \delta_{max}=8, \gamma_1=0.5, \gamma_2=2, \nu=10^{-4}, \eta=0.25,\varepsilon=10^{-4}.$}}}	
\label{mnistfig}
\end{figure}

\begin{figure}[H]
\centering
\includegraphics[width=5.7cm]{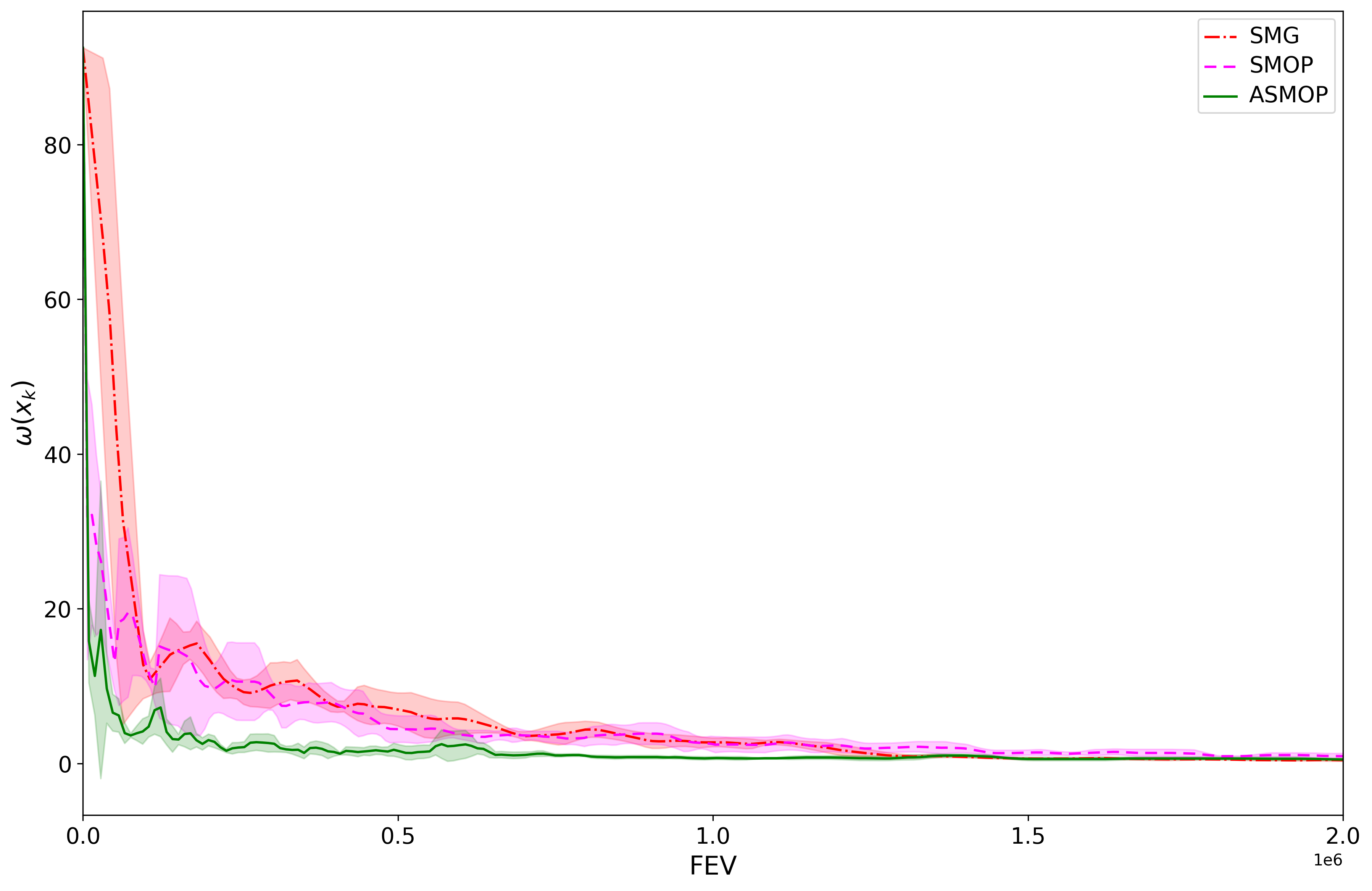}\quad
\includegraphics[width=5.7cm]{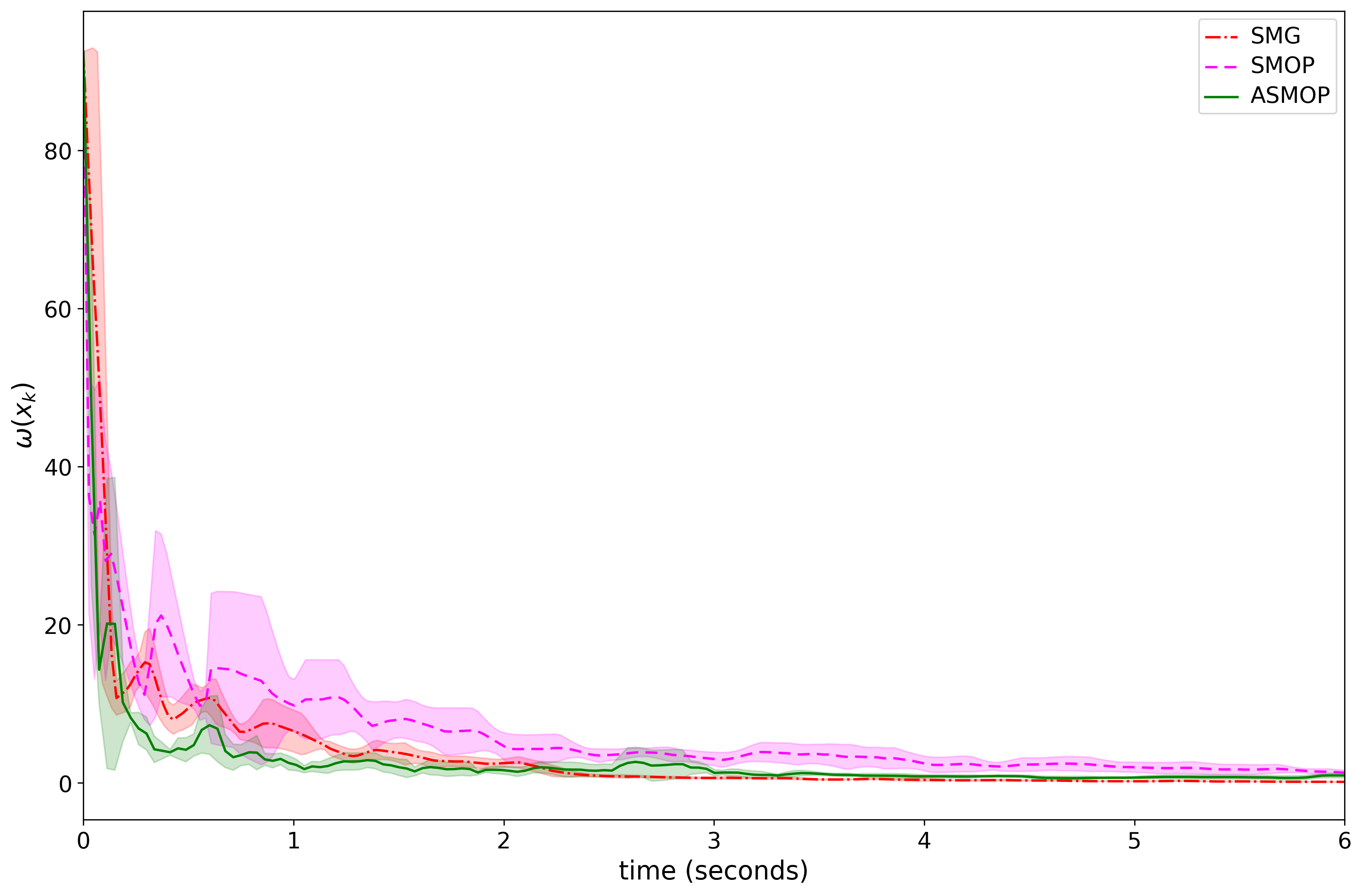}\\
\includegraphics[width=6.27cm]{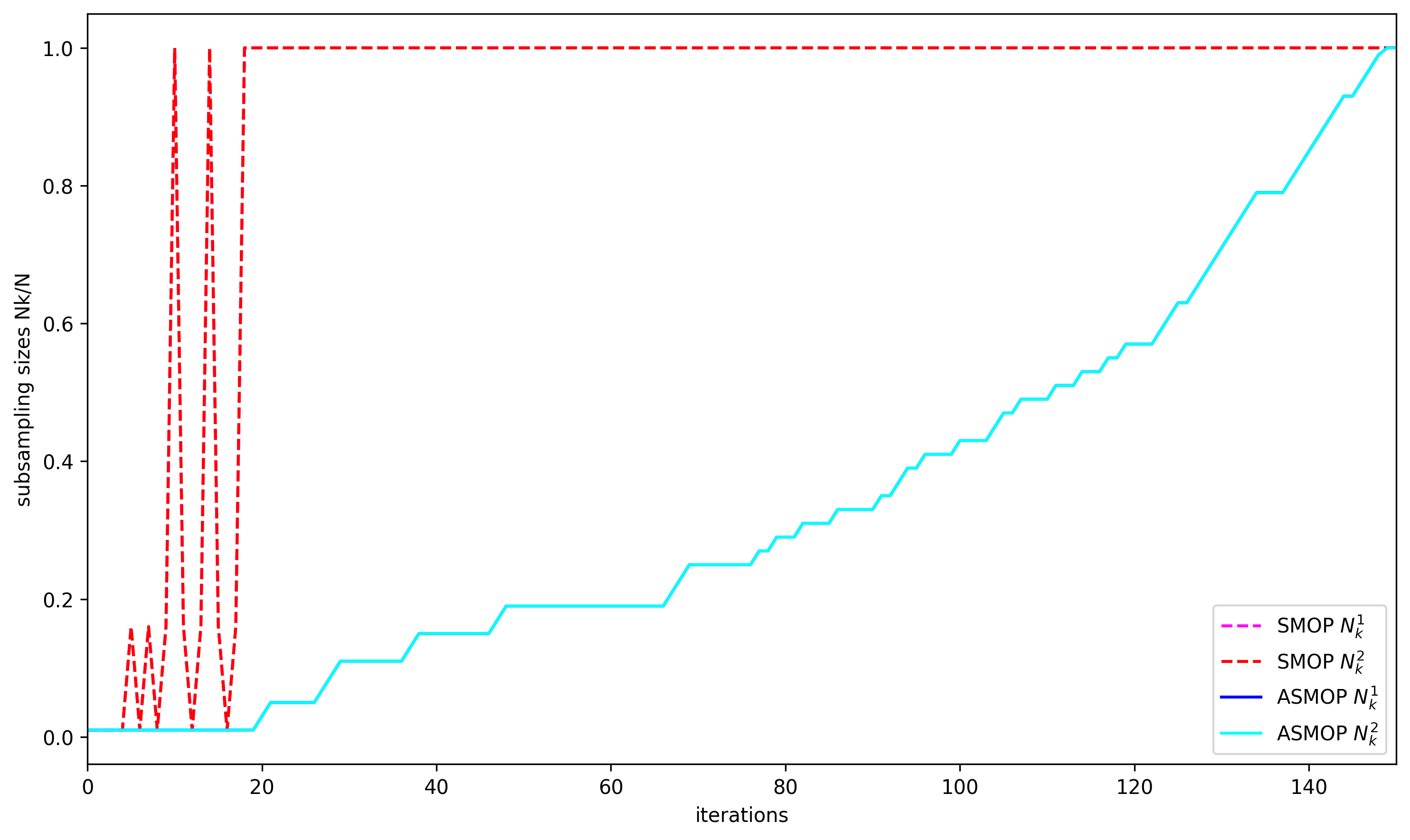}

\caption{{\footnotesize{Fashion MNIST dataset, problem \eqref{logregf},  $N=10^4,n=1024$. First row: optimality measure against function evaluations (left) and optimality measure against time in second (left). Second row: sample sizes behavior. Parameters: $x_0=(0.1,0.1,...,0.1), \delta_0=1, \delta_{max}=8, \gamma_1=0.5, \gamma_2=2, \nu=10^{-4}, \eta=0.25,\varepsilon=10^{-4}.$}}}	
\label{mnistfashionfig}
\end{figure}

\begin{figure}[H]
\centering
\includegraphics[width=5.7cm]{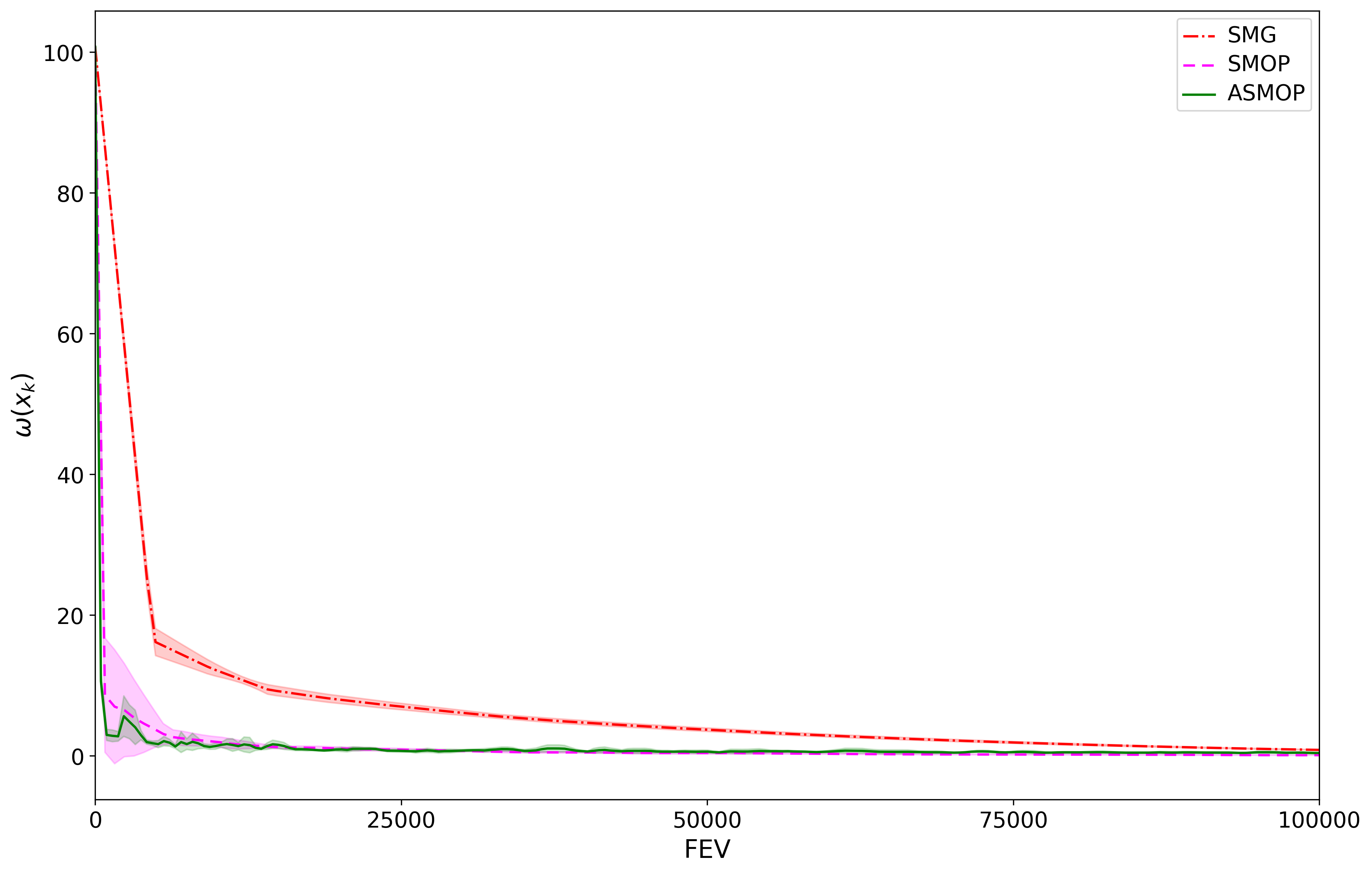}\quad
\includegraphics[width=5.7cm]{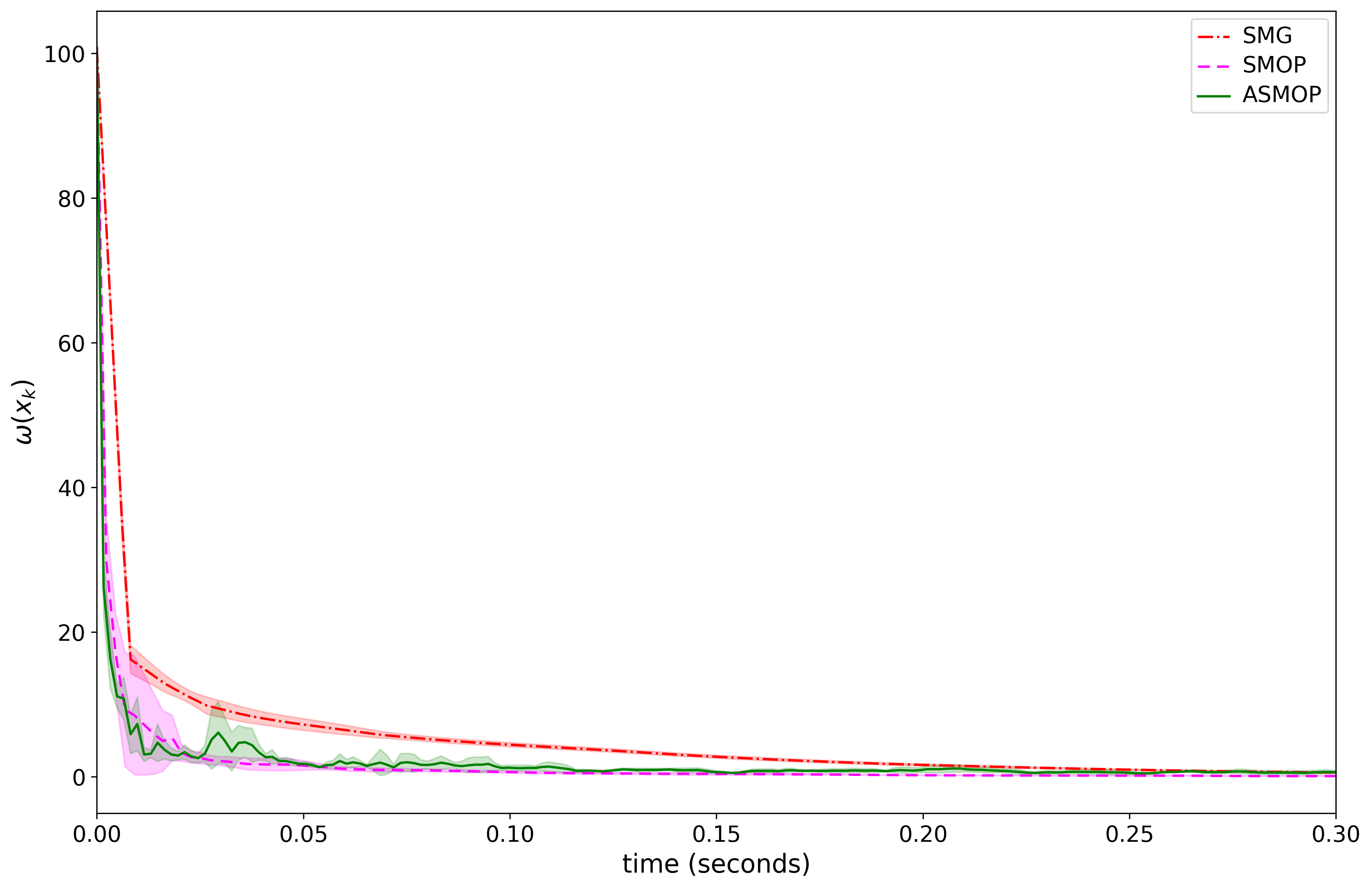}\\
\includegraphics[width=6.27cm]{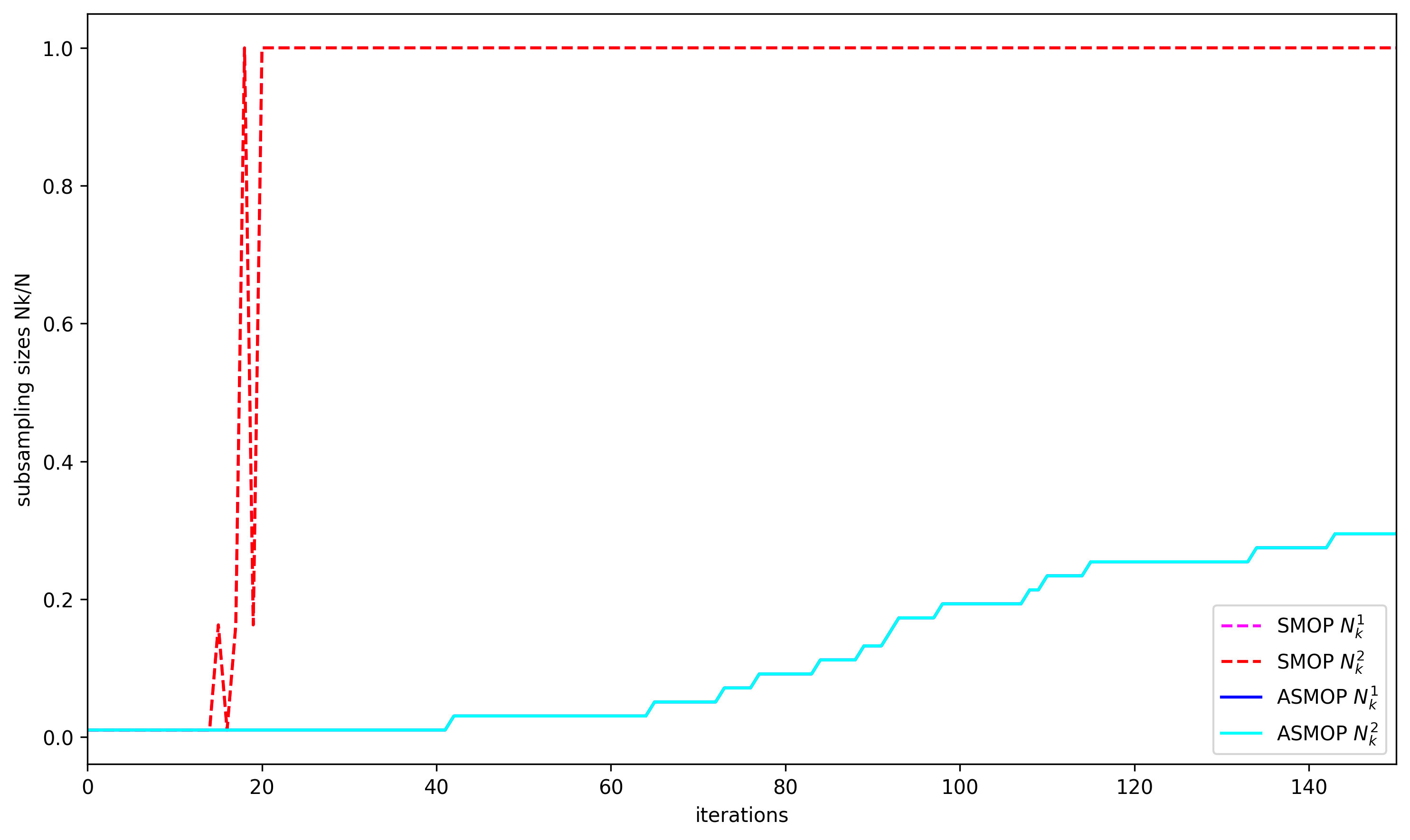}

\caption{{\footnotesize{MNIST-Fairness dataset, problem \eqref{logregf},  $N=10^4,n=1024$. First row: optimality measure against function evaluations (left) and optimality measure against time in second (left). Second row: sample sizes behavior. Parameters: $x_0=(0.1,0.1,...,0.1), \delta_0=1, \delta_{max}=8, \gamma_1=0.5, \gamma_2=2, \nu=10^{-4}, \eta=0.25,\varepsilon=10^{-4}.$}}}	
\label{mnistfairnessfig}
\end{figure}

\subsection{Regularized Nonconvex loss in 2-layer Neural Networks and  Least squares}

We have also tested how the ASMOP behaves when the criteria are two completely different loss functions. Once again we have used CIFAR10, MNIST, Fashion MNIST and MNIST-Fairness dataset for an image classification problem. The problem we are solving is \eqref{logreg}, however this time the component functions are
$$f^1(x)=\frac{1}{N}\sum_{j\in \mathcal{N}^1}\Big(1-\frac{1}{1+e^{(-y_j(x^Ta_j))}}\Big)^2+\frac{\lambda_1}{2}\|\hat{x}\|^2$$
and 
\begin{equation}
\label{mnst}
    f^2(x)=\frac{1}{N}\sum_{j\in \mathcal{N}^2}\frac{1}{2} (x^Ta_j-y_j)^2
\end{equation}
where $N$ is the size of the respective sample groups, $x\in \mathbb{R}^n$ is the vector of model coefficients, $\hat{x}$ coefficient vector without the intercept,  $a_j$ the attribute vector of the sample and $y_j$ its respective label. By minimizing this loss function, we get coefficients that are adjusted for both machine learning models. Specifically, for CIFAR10 it will use a regularized Nonconvex loss in 2-layer Neural Networks to differentiate images of cars and planes, and weighted least squares method to differentiate birds from cats, and analogously for the other three datasets considered. In the similar manner, we set the initial subsampling sizes $N_0^i=0.01N$, and the step size update to $\Delta N_k^i=0.02N$. The predetermined nonmonotonicity sequences were set to $t_k=\frac{1}{(k+1)^{1.51}}$ and $\overline{t}_k=\frac{100}{(k+1)^{1.51}}$ as in the previous experiment. The following figures show $\omega(x_k)$ values in terms of number of function evaluations for a fixed budget of $7\cdot10^6$ for CIFAR10, $10^6$ for MNIST, $10^7$ for Fashion MNIST and $10^5$ for MNIST-Fairness dataset. We also plot $\omega(x_k)$  against the execution time for a fixed budget of 100 seconds for CIFAR10, 5 seconds for MNIST, 25 seconds for Fashion MNIST and 2 seconds for MNIST-Fairness. In all Figures \ref{cifar2}, \ref{mnistfig2}, \ref{mnistfashionfig2} and \ref{mnistfairnessfig2} it can be seen that for the given budget ASMOP shows efficiency and a large decrease in $\omega(x_k)$ value for a small cost.
Also in this case, five runs are performed and the reported curves correspond to the average results for each method. In the plots of $\omega(x_k)$, the shaded region represents the Standard Deviation.

\begin{figure}[H]
\centering
\includegraphics[width=5.7cm]{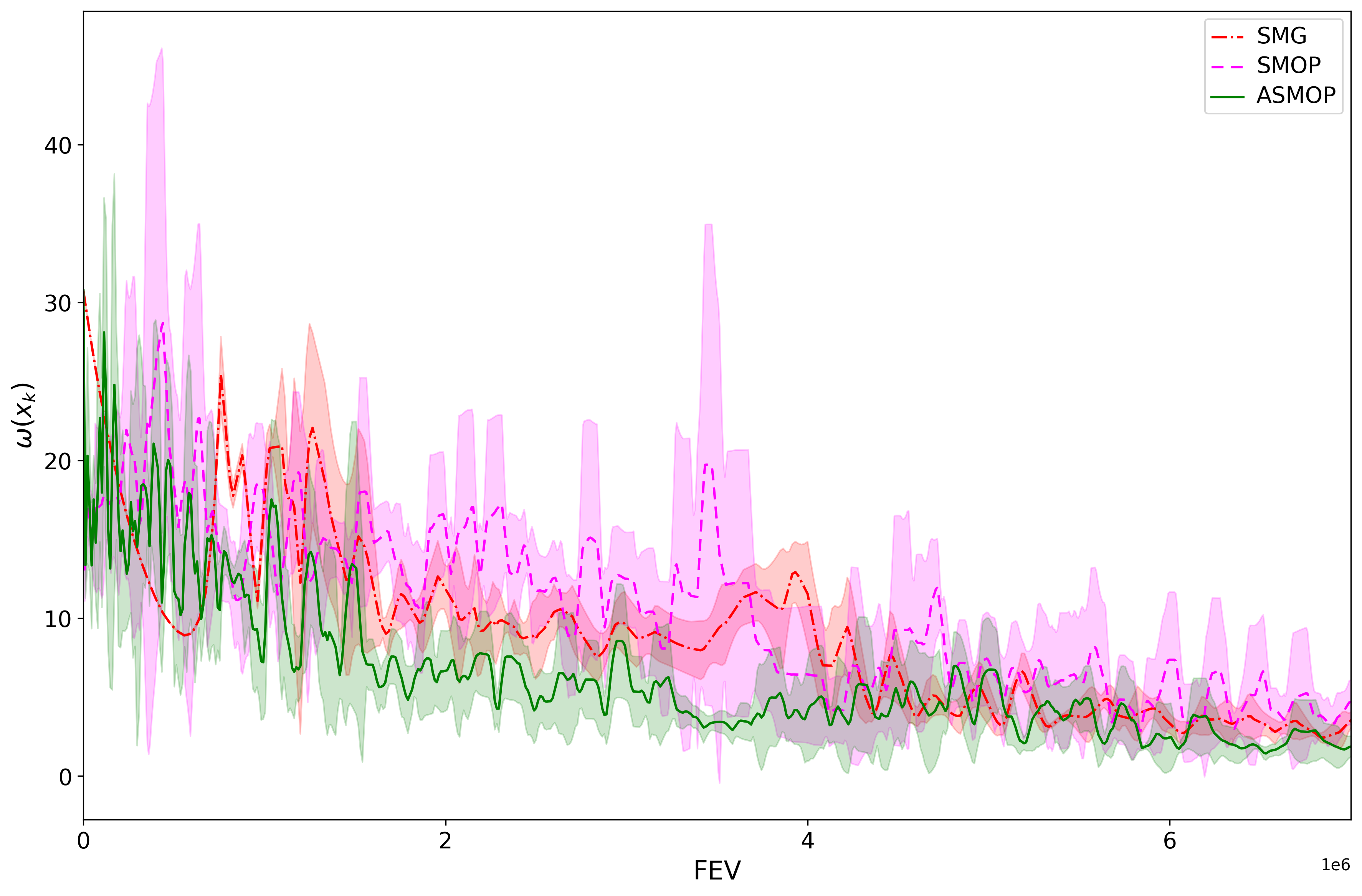}\quad
\includegraphics[width=5.7cm]{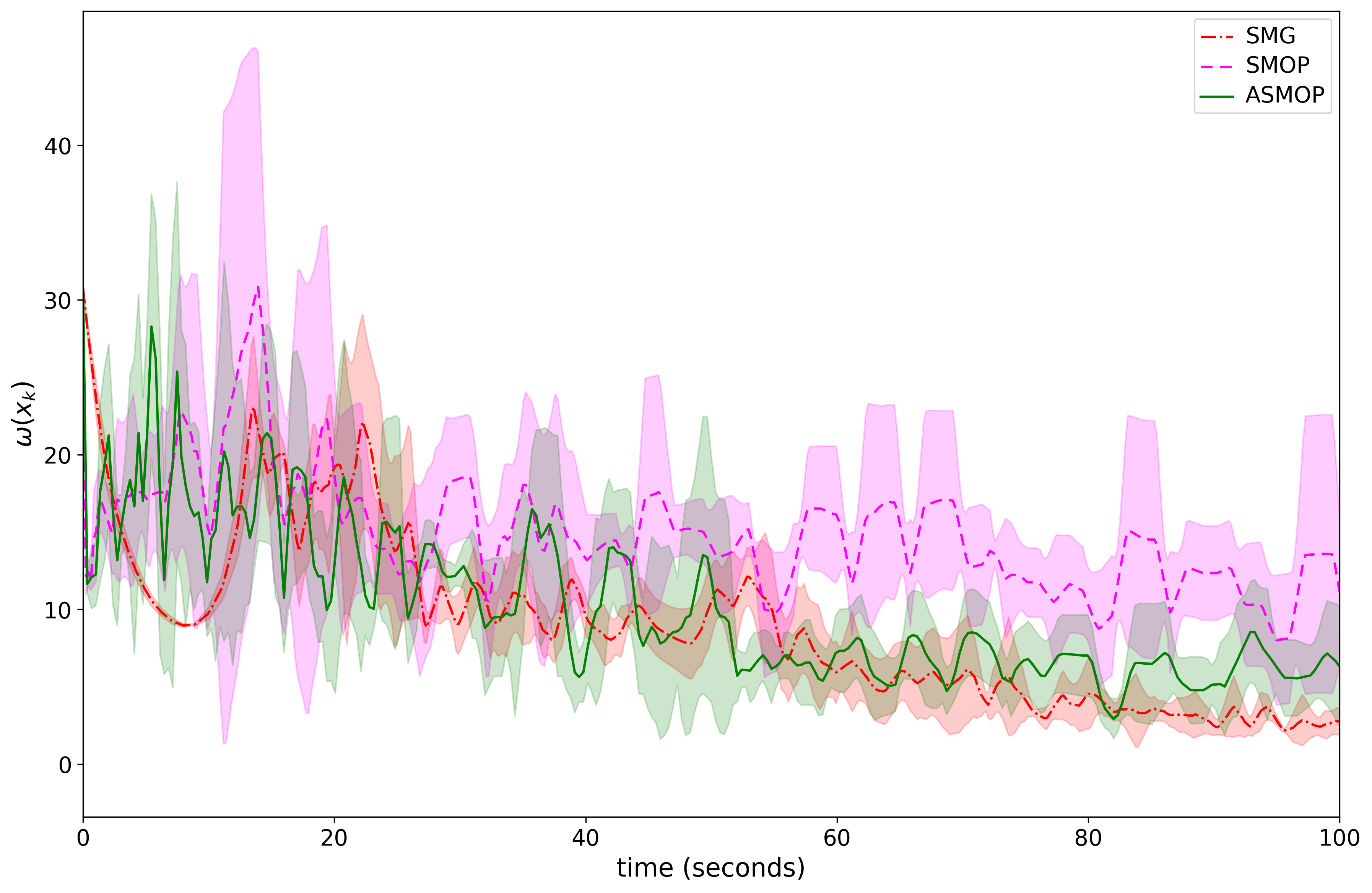}\\
\includegraphics[width=6.27cm]{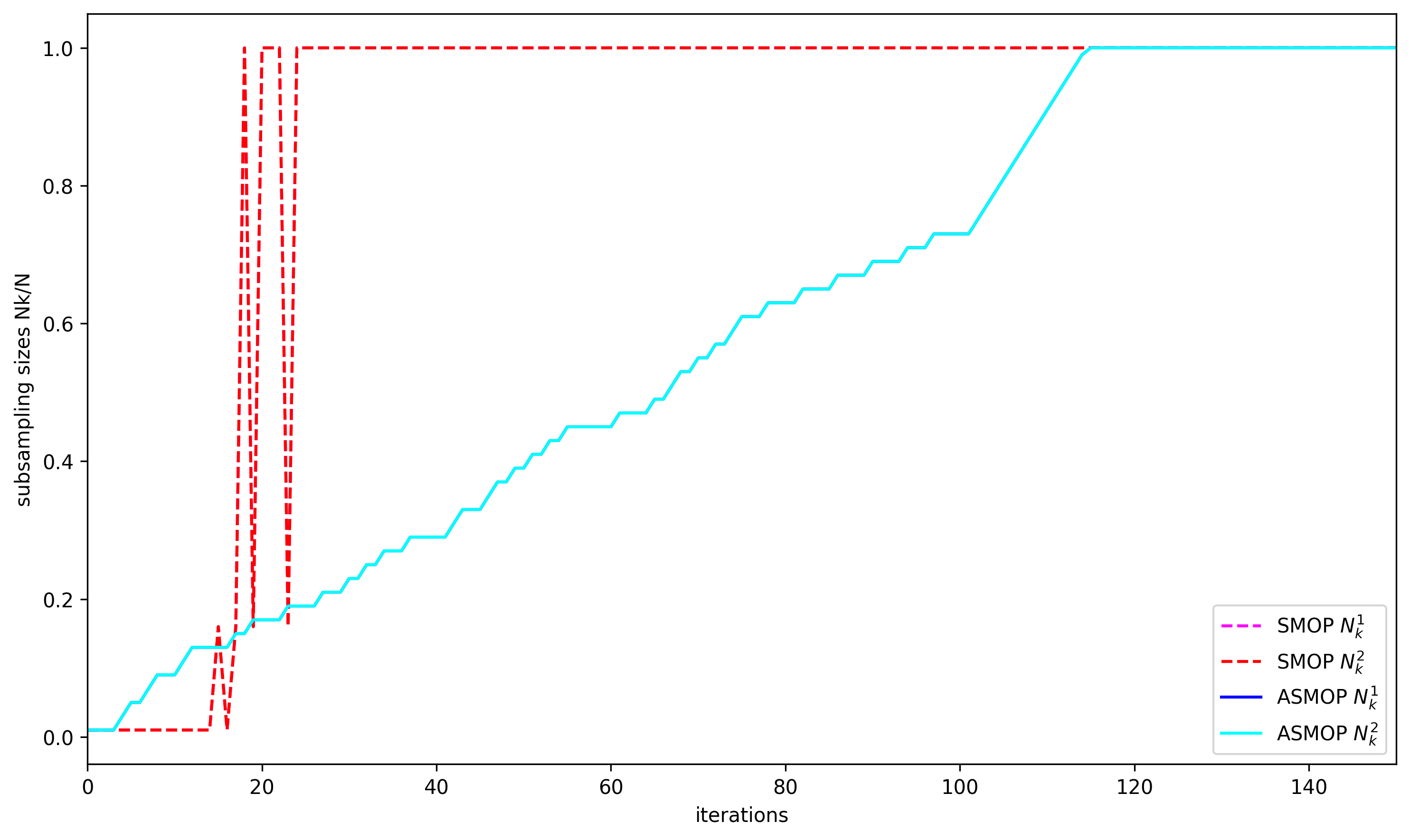}
\caption{{\footnotesize{CIFAR10 dataset, problem \eqref{mnst}, $N=10^4,n=3072$. First row: optimality measure against function evaluations (left) and optimality measure against time in second (left). Second row: sample sizes behavior. Parameters: $x_0=(0.1,0.1,...,0.1), \delta_0=1, \delta_{max}=8, \gamma_1=0.5, \gamma_2=2, \nu=10^{-5}, \eta=0.25,\varepsilon=10^{-5}.$}}}	
\label{cifar2}
\end{figure}

\begin{figure}[H]
\centering
\includegraphics[width=5.7cm]{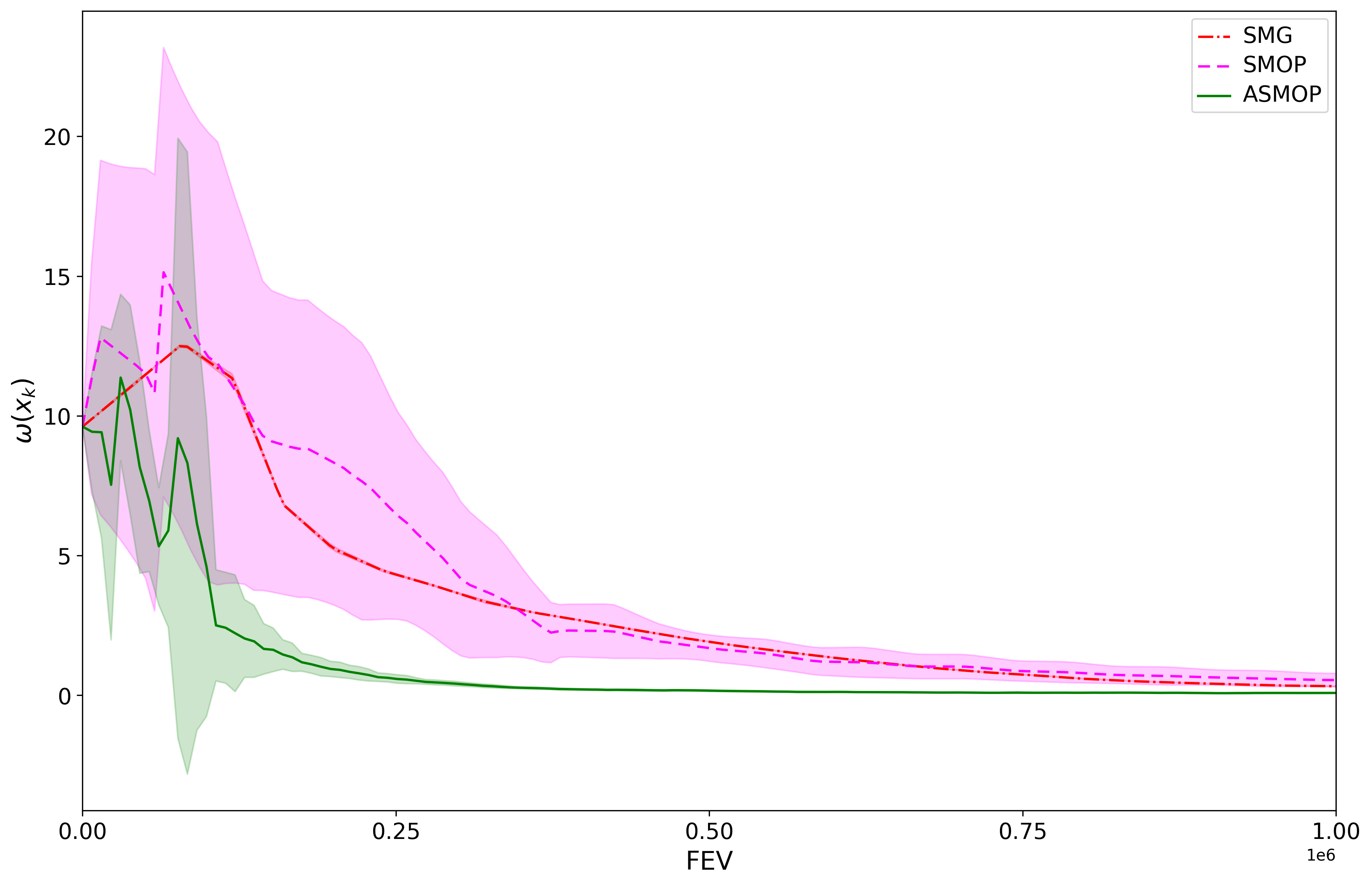}\quad
\includegraphics[width=5.7cm]{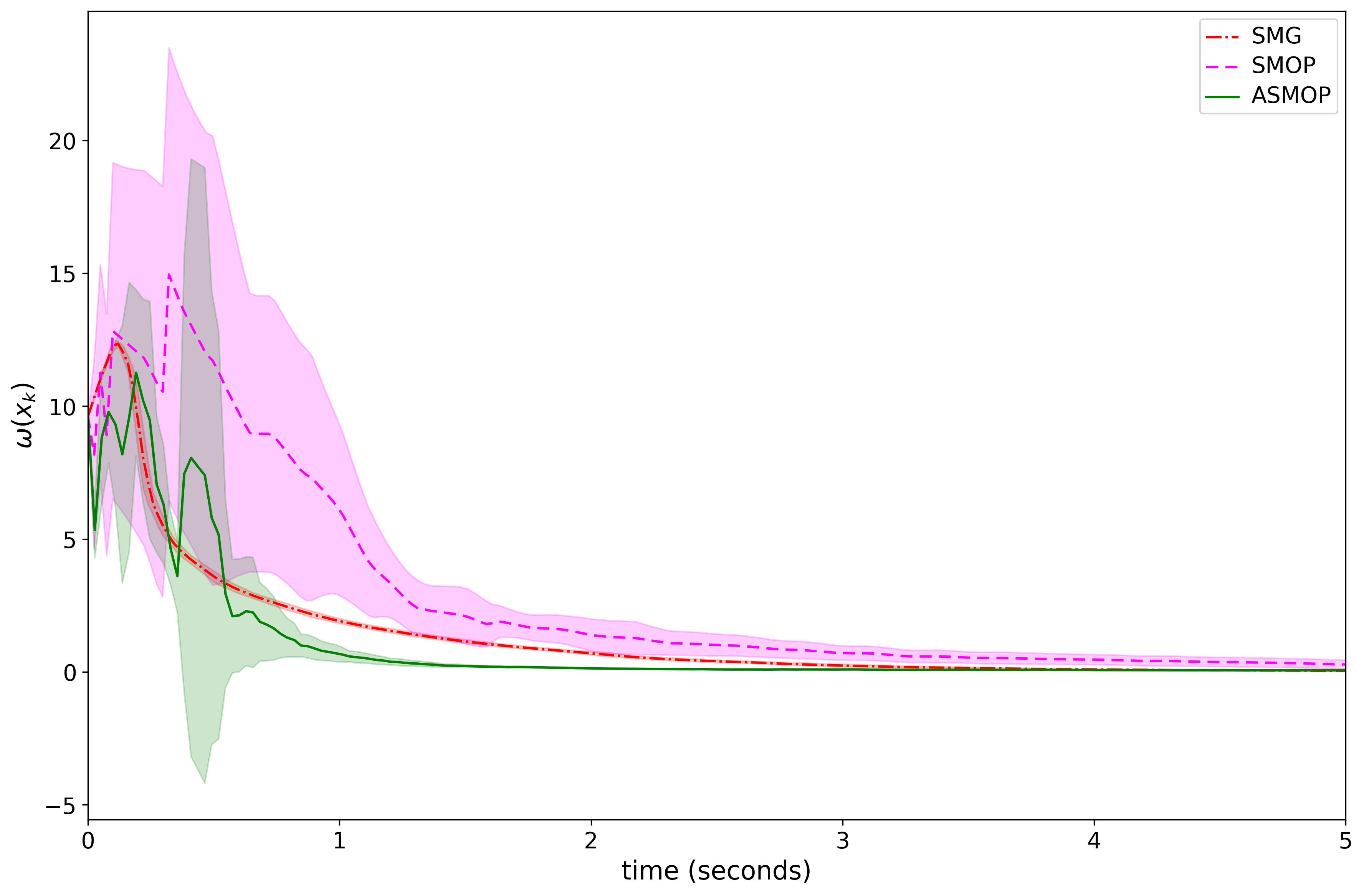}\\
\includegraphics[width=6.27cm]{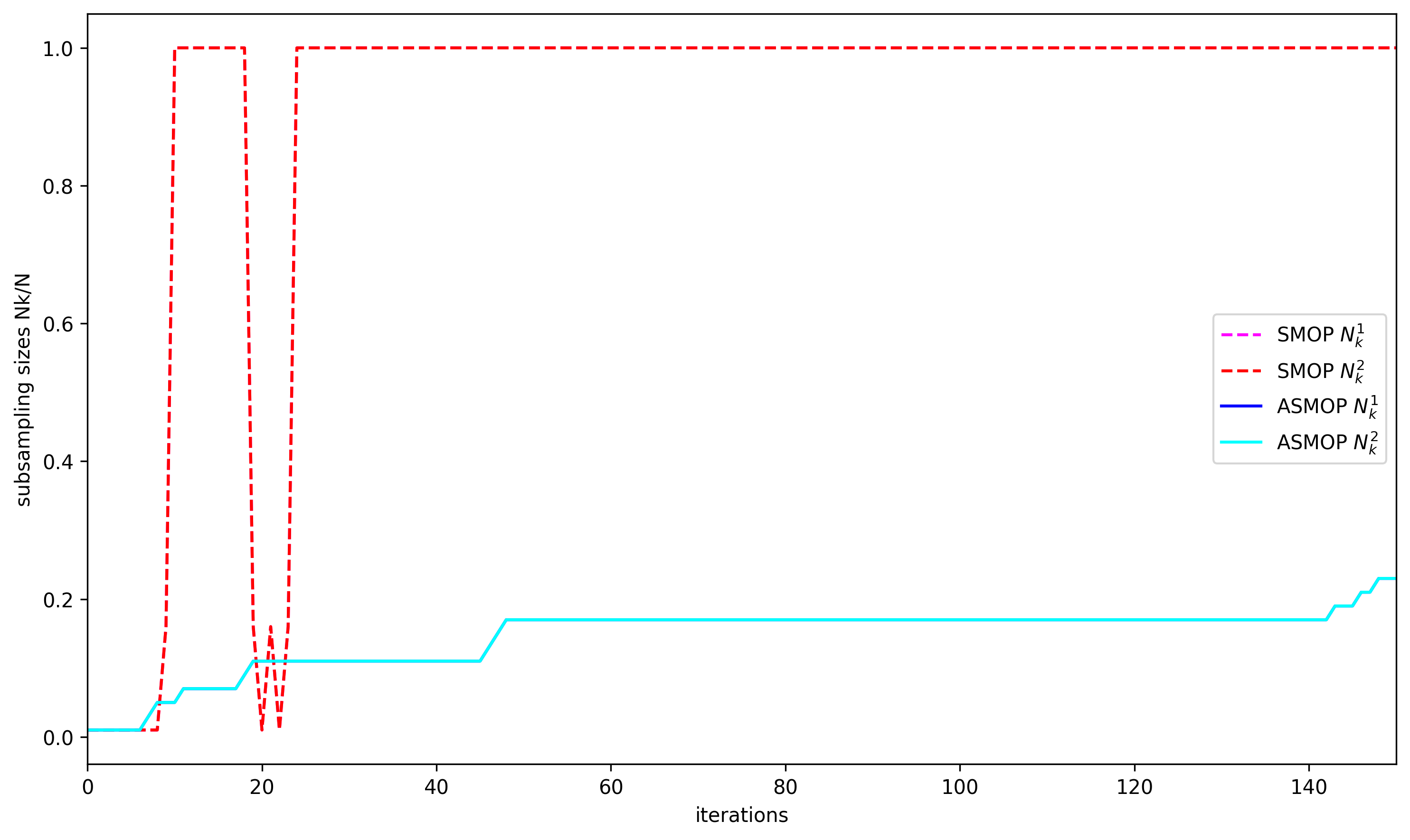}

\caption{{\footnotesize{MNIST dataset, problem \eqref{mnst},  $N=10^4,n=1024$. First row: optimality measure against function evaluations (left) and optimality measure against time in second (left). Second row: sample sizes behavior. Parameters: $x_0=(0.1,0.1,...,0.1), \delta_0=1, \delta_{max}=8, \gamma_1=0.5, \gamma_2=2, \nu=10^{-5}, \eta=0.05,\varepsilon=10^{-5}.$}}}	
\label{mnistfig2}
\end{figure}

\begin{figure}[H]
\centering
\includegraphics[width=5.7cm]{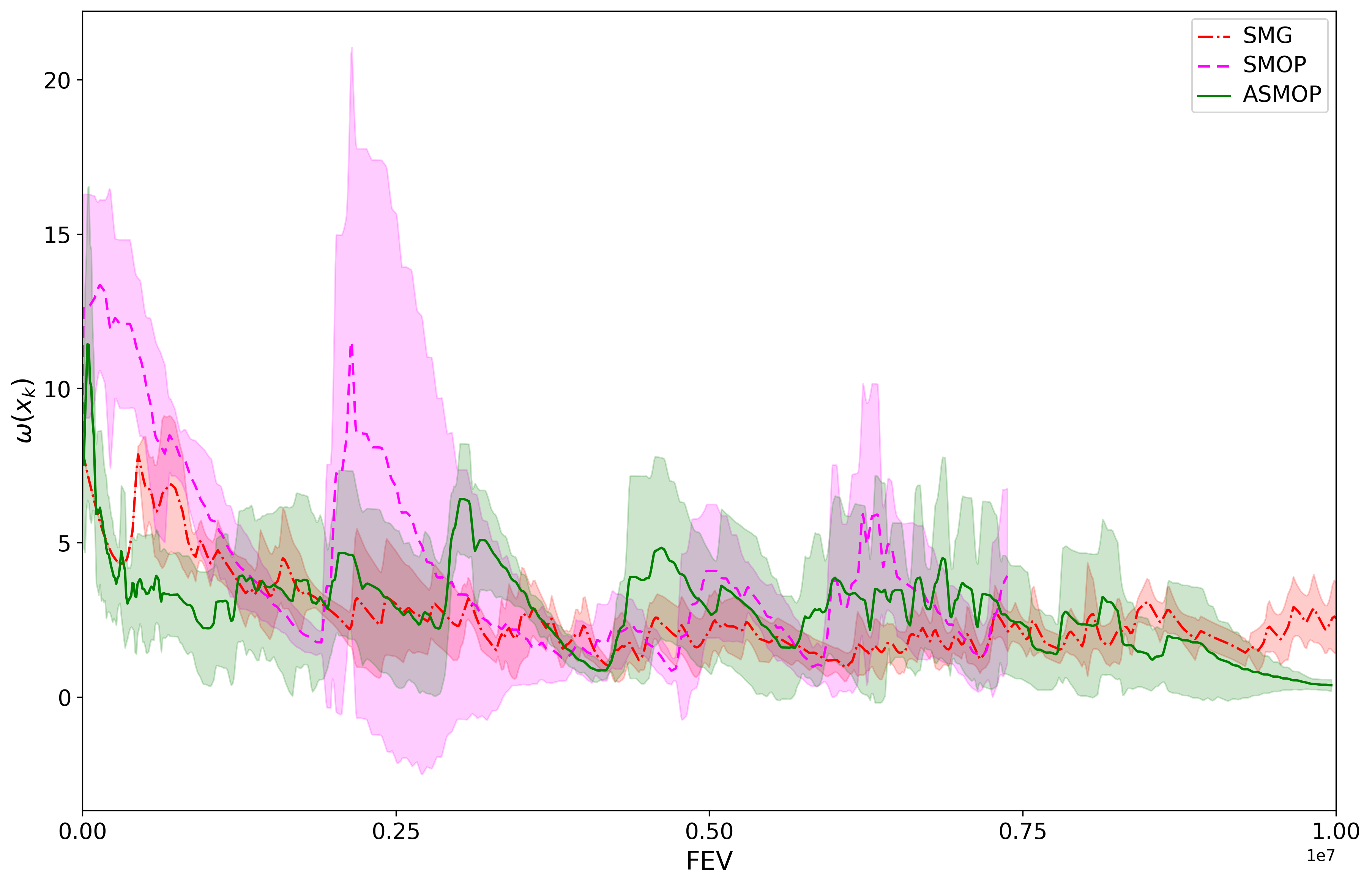}\quad
\includegraphics[width=5.7cm]{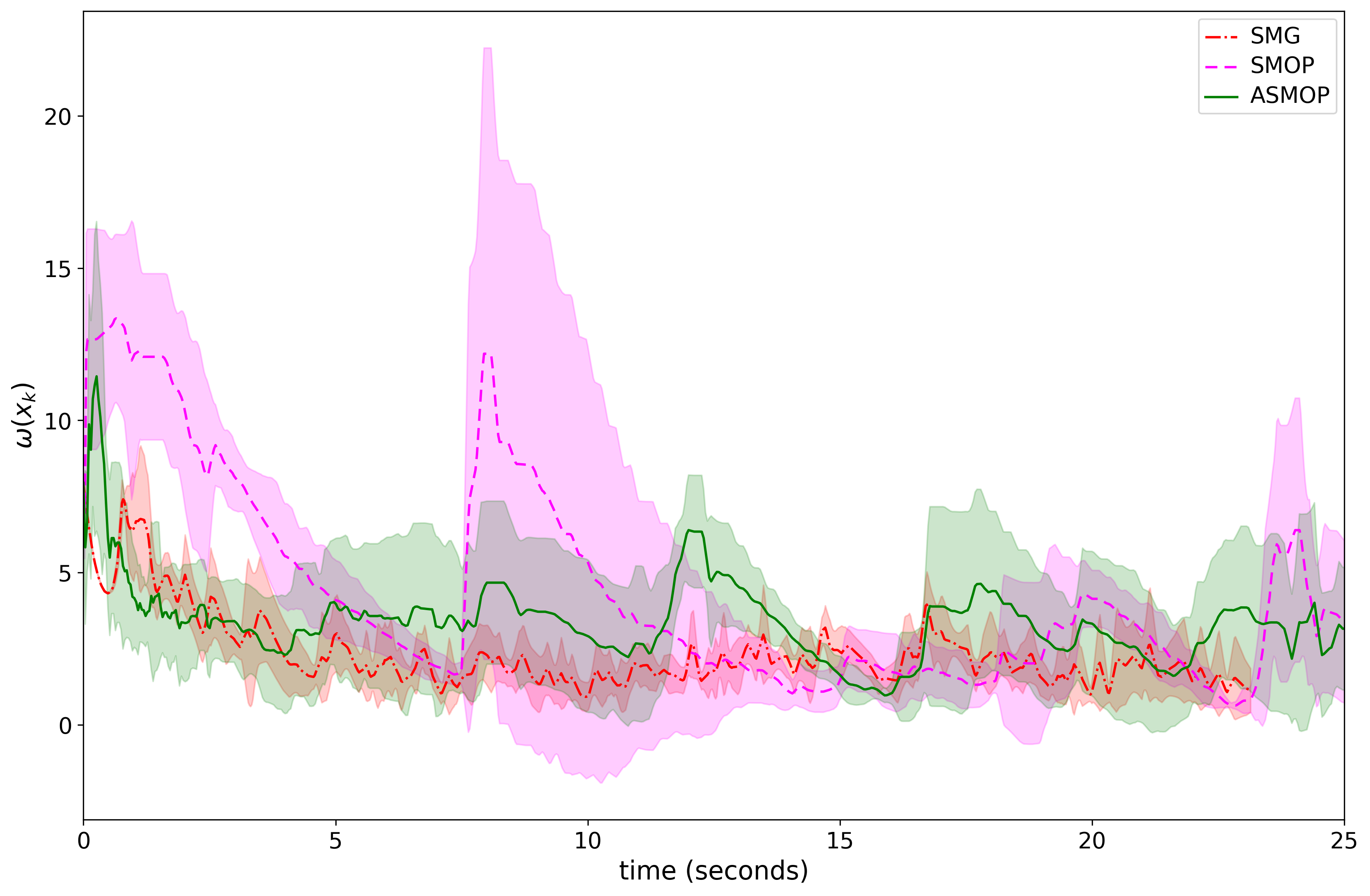}\\
\includegraphics[width=6.27cm]{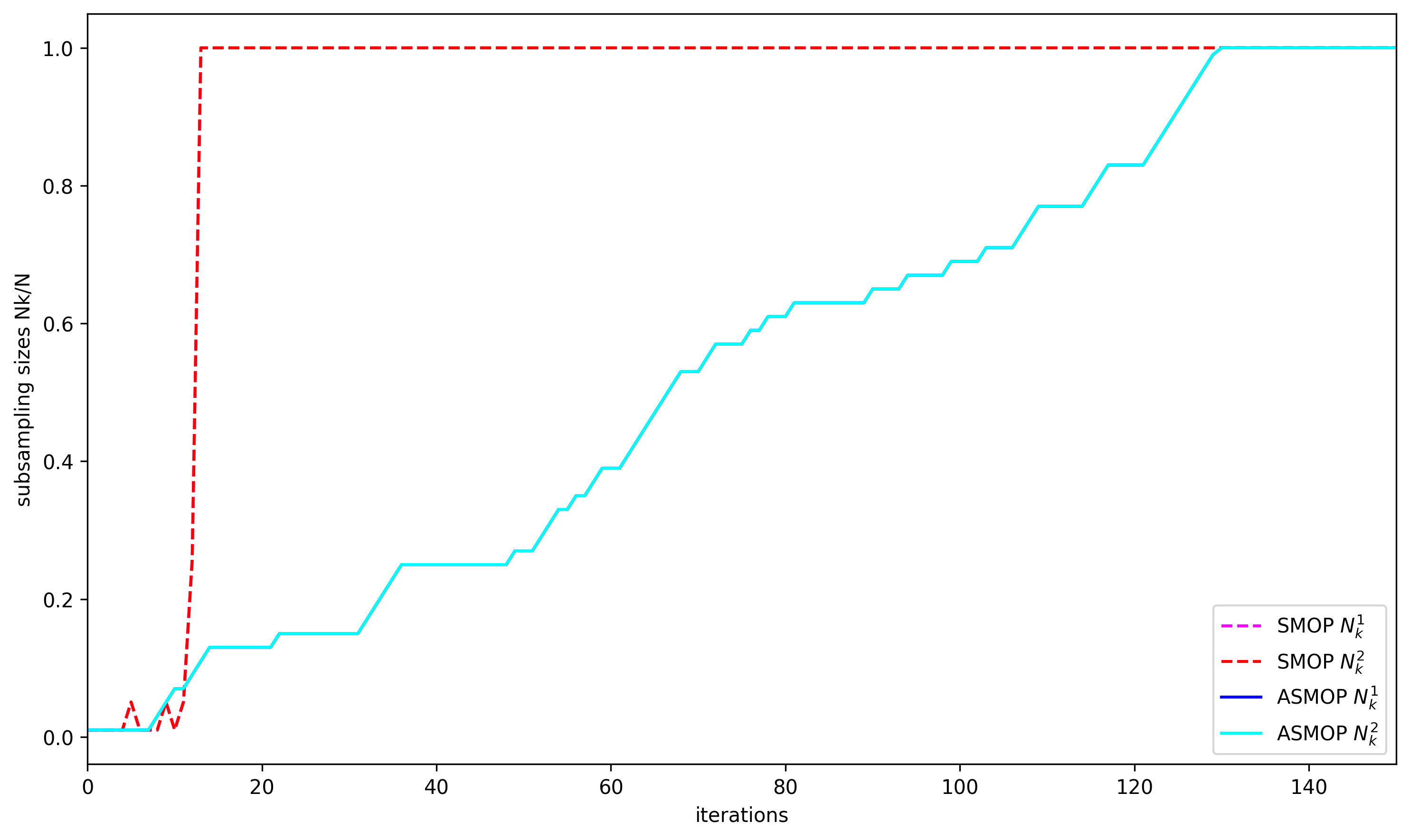}

\caption{{\footnotesize{Fashion MNIST dataset, problem \eqref{mnst},  $N=10^4,n=1024$. First row: optimality measure against function evaluations (left) and optimality measure against time in second (left). Second row: sample sizes behavior. Parameters: $x_0=(0.1,0.1,...,0.1), \delta_0=1, \delta_{max}=8, \gamma_1=0.6, \gamma_2=1.5, \nu=10^{-5}, \eta=0.01,\varepsilon=10^{-4}.$}}}	
\label{mnistfashionfig2}
\end{figure}

\begin{figure}[H]
\centering
\includegraphics[width=5.7cm]{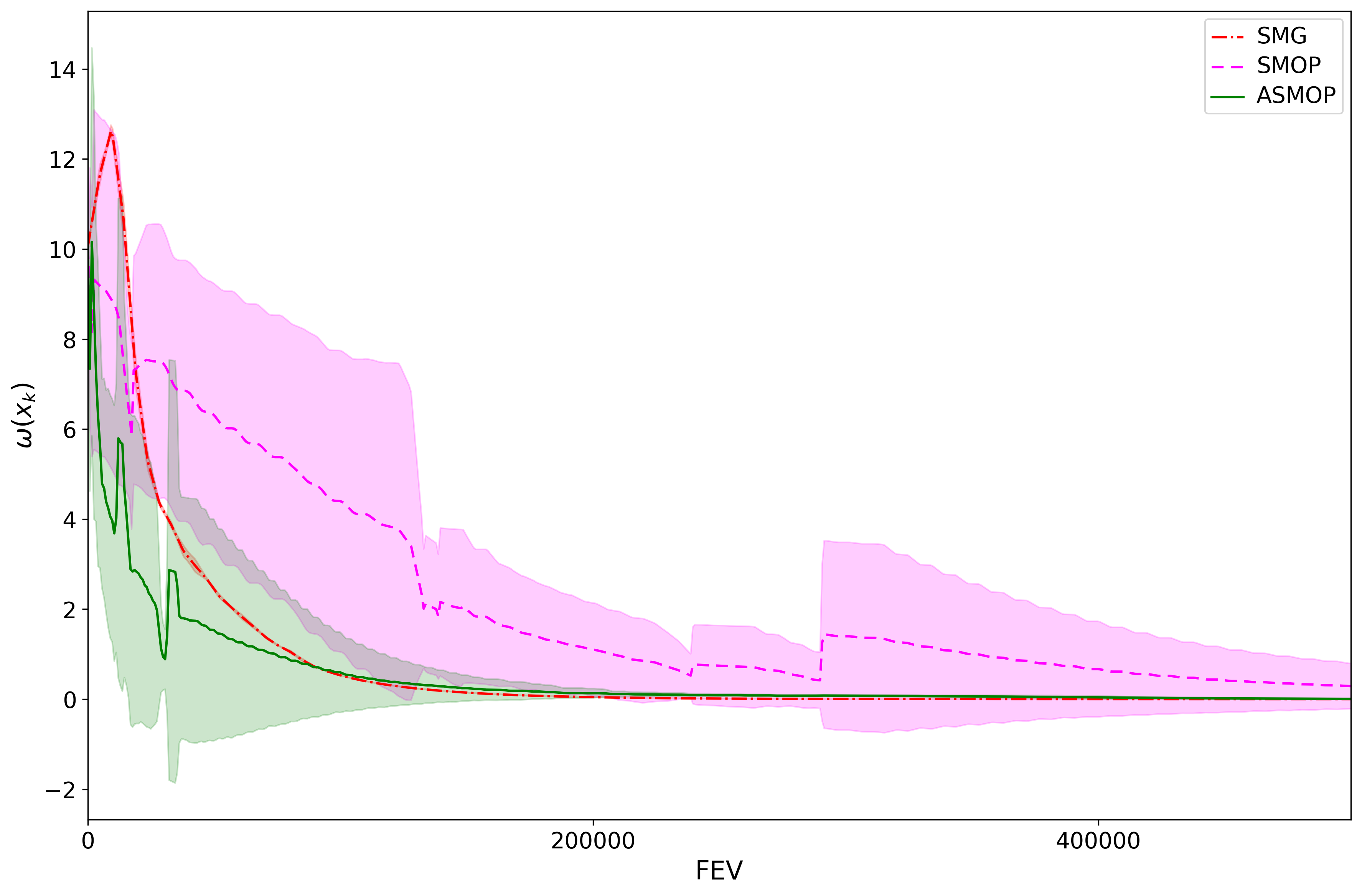}\quad
\includegraphics[width=5.7cm]{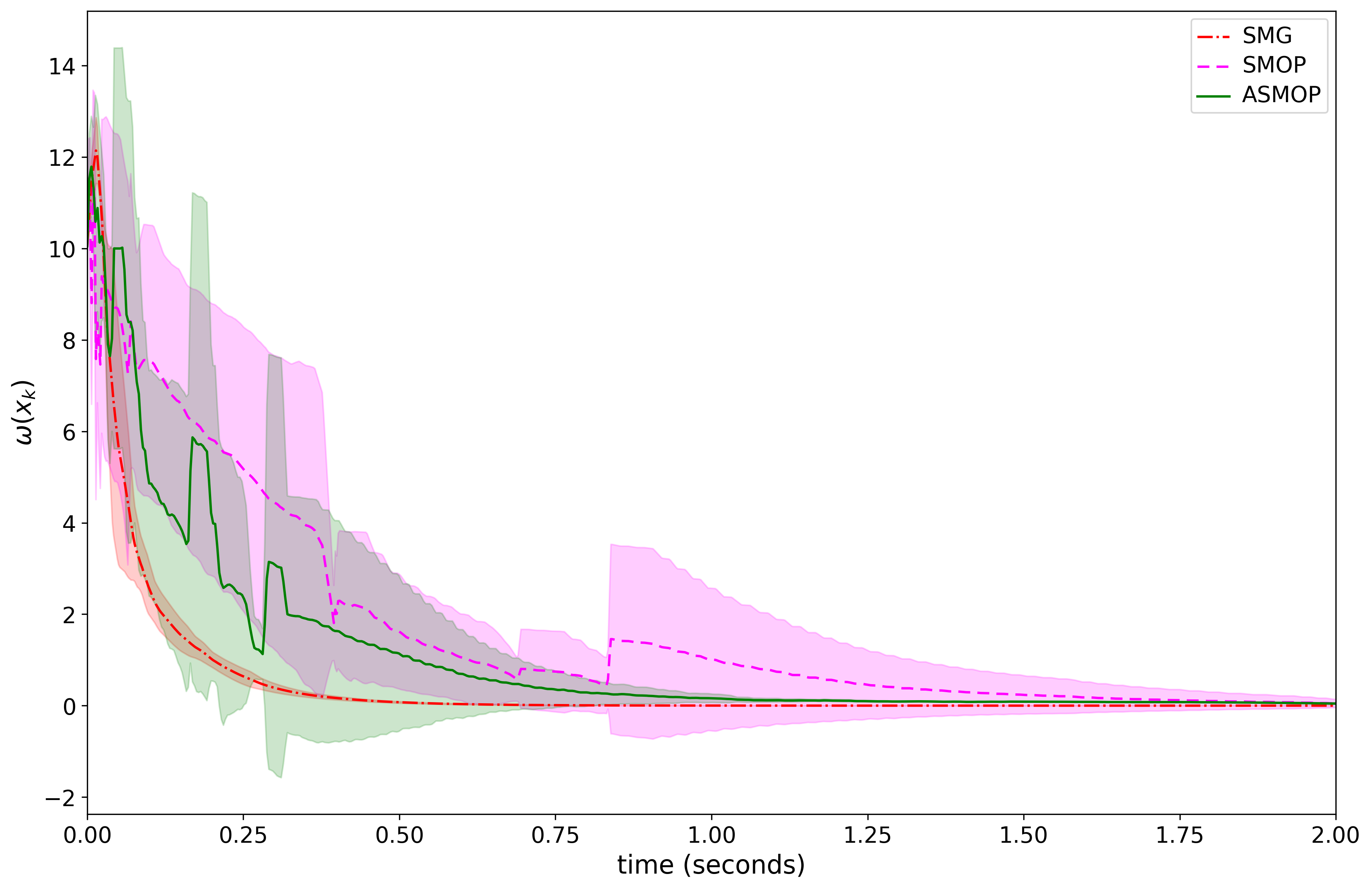}\\
\includegraphics[width=6.27cm]{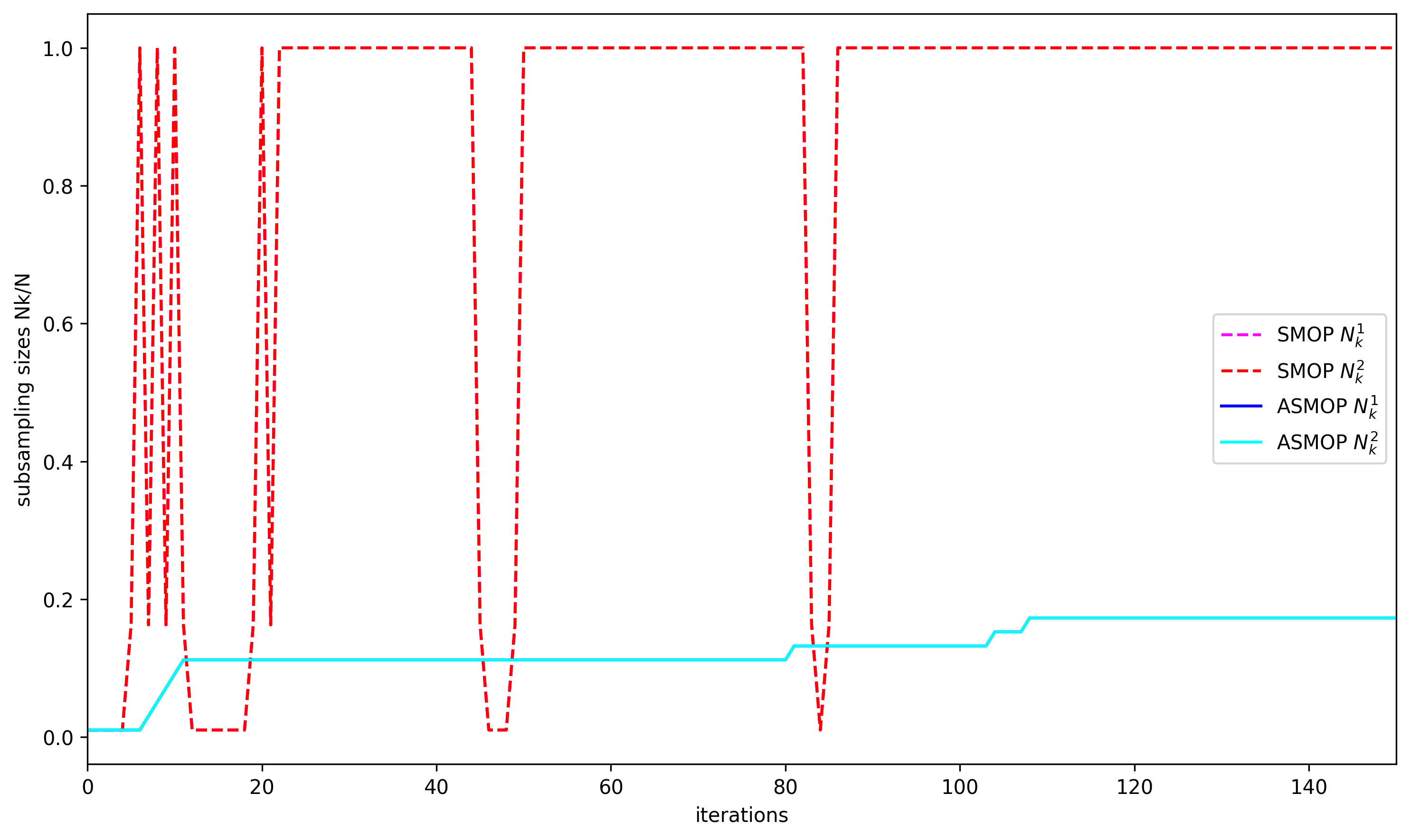}

\caption{{\footnotesize{MNIST-Fairness dataset, problem \eqref{mnst},  $N=10^4,n=1024$. First row: optimality measure against function evaluations (left) and optimality measure against time in second (left). Second row: sample sizes behavior. Parameters: $x_0=(0.1,0.1,...,0.1), \delta_0=1, \delta_{max}=8, \gamma_1=0.5, \gamma_2=2, \nu=10^{-5}, \eta=0.05,\varepsilon=10^{-5}.$}}}	
\label{mnistfairnessfig2}
\end{figure}

\subsection{Nonmonotonicity parameters and increasing rule}
In the previous experiments, we set parameters $t_k$ as $\frac{1}{(k+1)^{1.51}}$ and $\overline{t}_k$ as $\frac{100}{(k+1)^{1.51}}$ . By adjusting these settings it is possible to increase or decrease the tolerance of the nonmonotonicity, which leads to different algorithm behavior. We set
$$\overline{t}_k=\frac{C_2}{(k+1)^{1.51}}$$
and tested three different scenarios ($C_2\in \{1, 100, 10000\}$) in order to see how the relaxation of the condition $\rho_{\cald}>\nu$ impacts the performance. The following table shows the chosen settings of the compared algorithms for MNIST dataset, whereas the rest of the parameters were set as in the second experiment. 
\begin{table}[H]
\begin{center}
    
\begin{tabular}{|c|c|c|}
\hline
     (MNIST)  &  $C_2$ \\ \hline
ASMOP$\_$1 &  $10000$     \\ \hline
ASMOP$\_$2 &  $1$     \\ \hline
ASMOP$\_$3 &  $100$       \\ \hline
\end{tabular}
\caption{\footnotesize{MNIST dataset. Different nonmonotonicity settings for ASMOP versions.}}
\label{tab1}
\end{center}

\end{table}
We compared these three algorithms similarly as in previous experiments, by criticality measure $\omega(x_k)$ in terms of function evaluations and time in seconds. The problem being solved is \eqref{logregf}.
Like the previous cases, five runs are performed and the reported curves correspond to the average results for each method. In the plots of $\omega(x_k)$, the shaded region represents the Standard Deviation.

\begin{figure}[H]
\centering
\includegraphics[width=5.7cm]{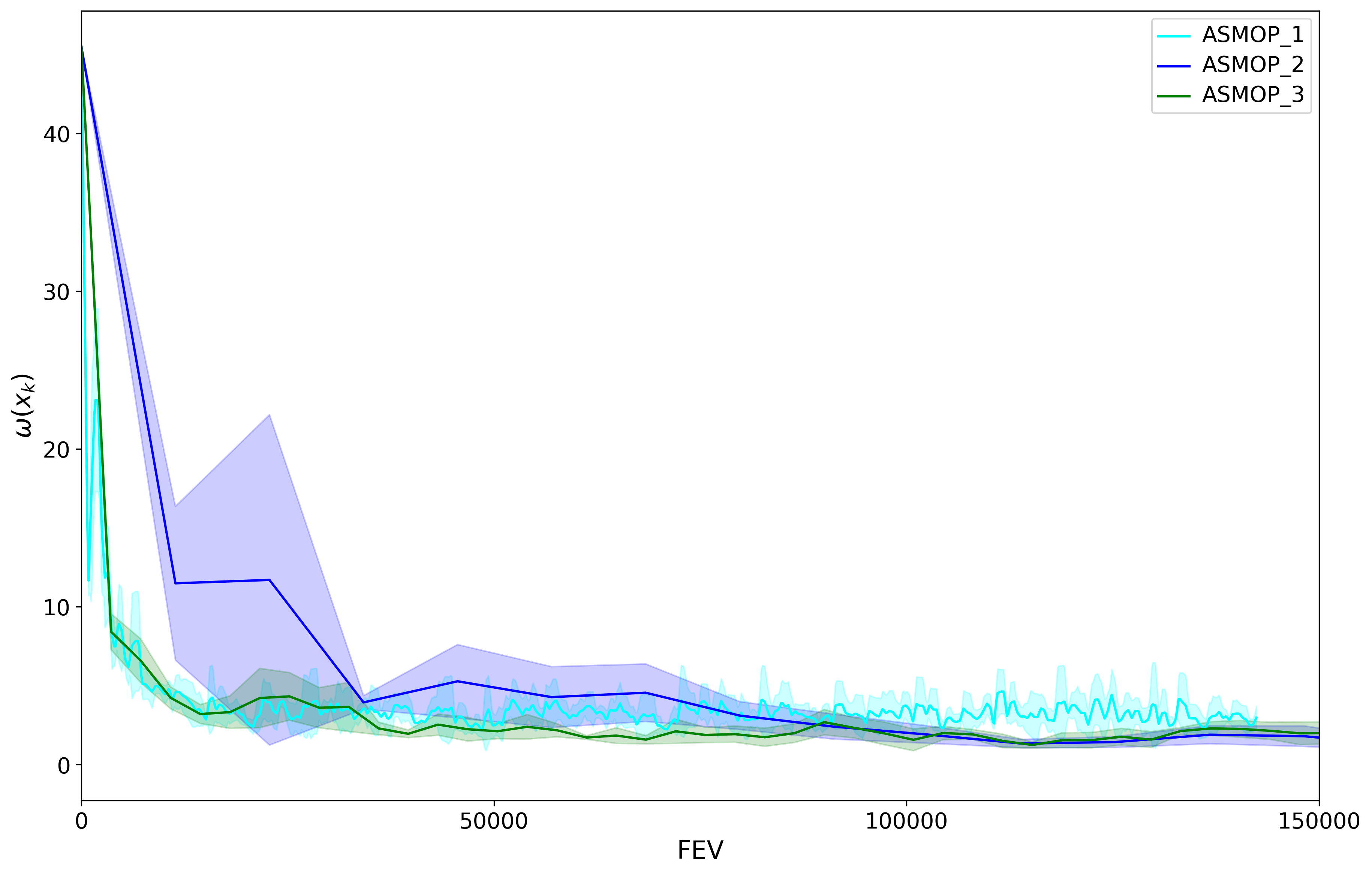}\quad
\includegraphics[width=5.7cm]{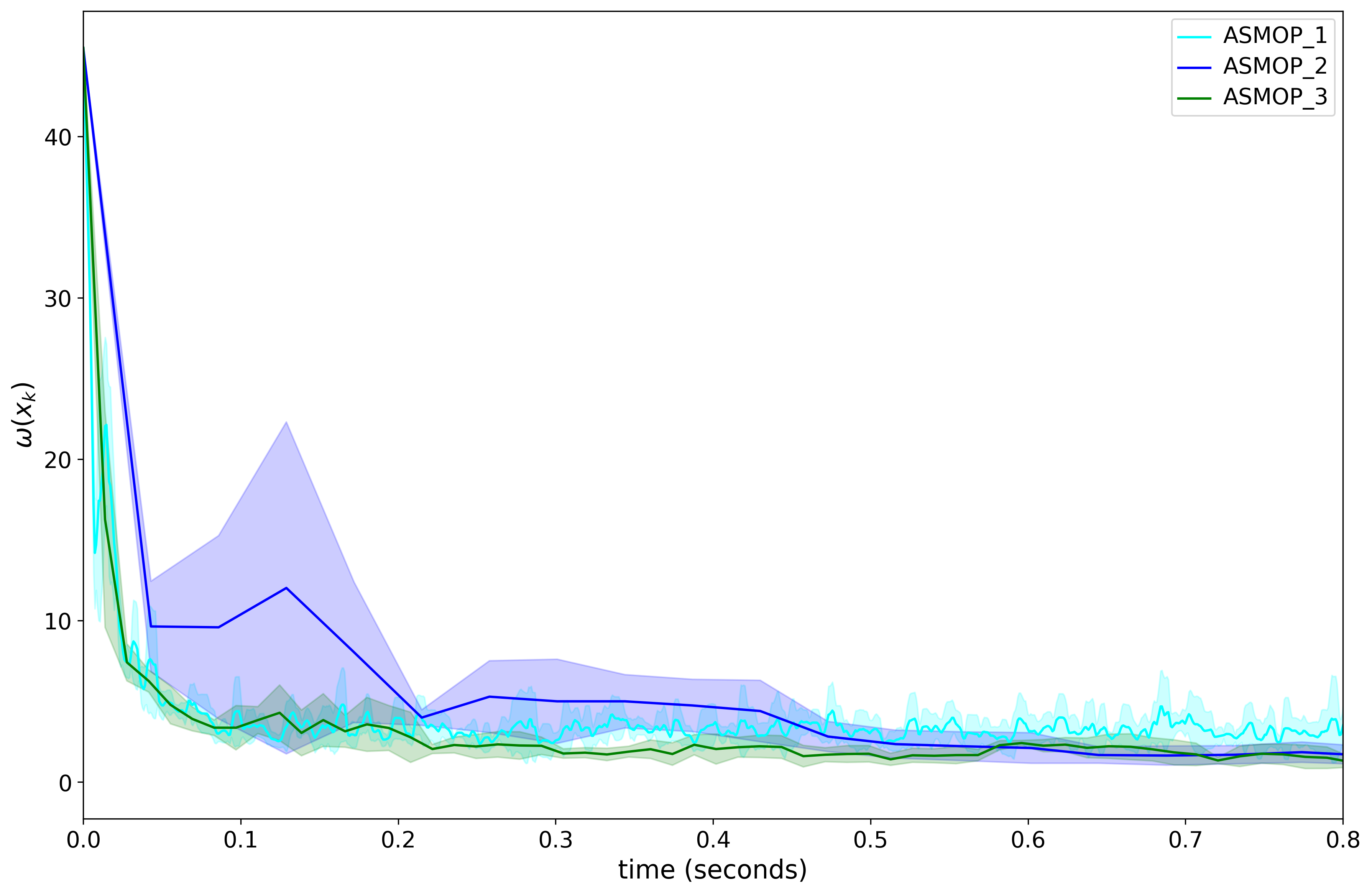}\\
\includegraphics[width=6.27cm]{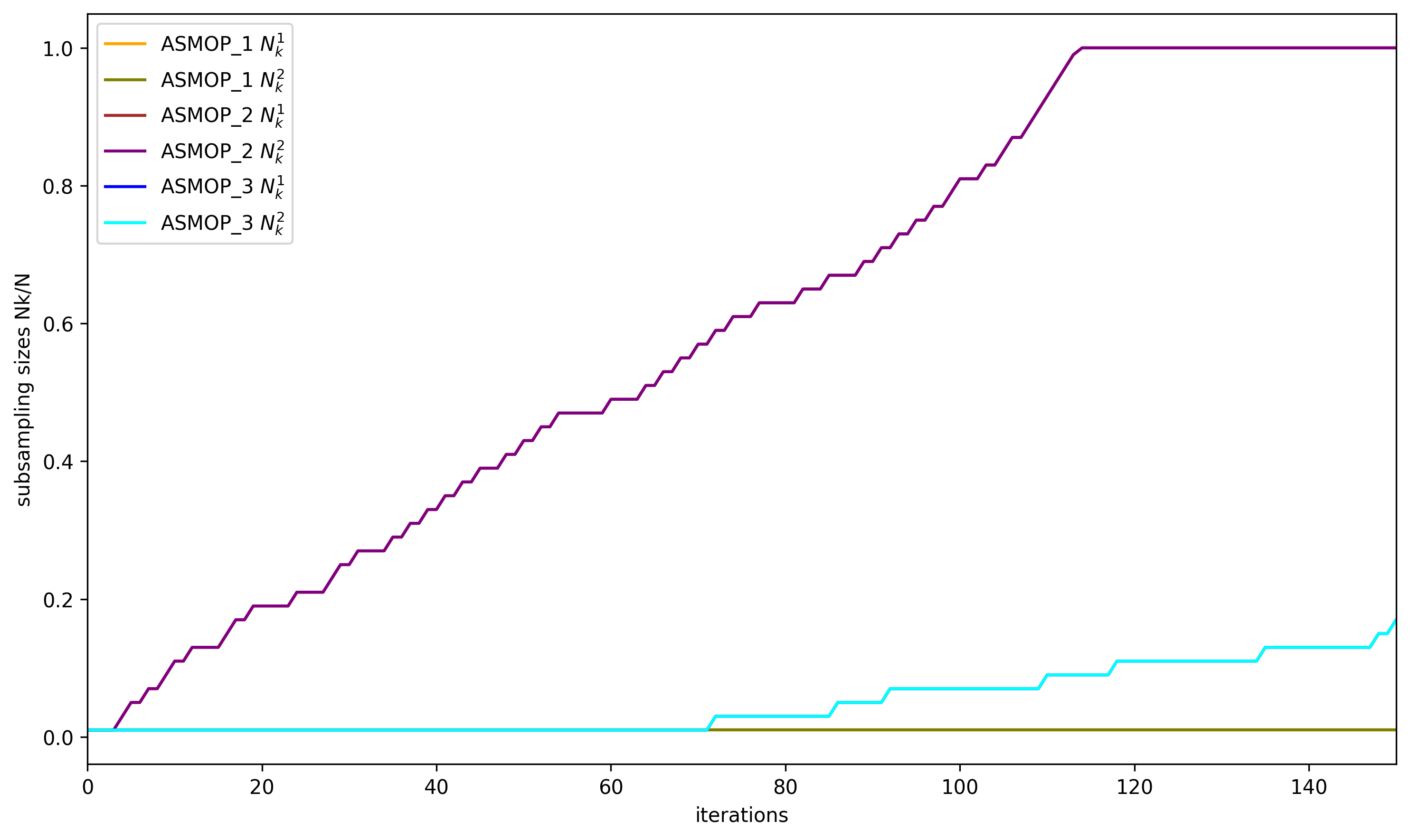}

\caption{{\footnotesize{MNIST dataset, problem \eqref{logregf} different settings Table \ref{tab1}, $N=10^4, n=1024$. Optimality measure against computational cost. Parameters: $x_0=(0.1,0.1,...,0.1), \delta_0=1, \delta_{max}=8, \gamma_1=0.5, \gamma_2=2, \nu=10^{-4}, \eta=0.25,\varepsilon=10^{-4}.$}}}	
\label{mncomp}
\end{figure}

It is noticeable that the subsampling sizes are increased less frequently for the versions that have a more relaxed coefficient $\rho_{\cald}$, which means that the higher tolerance leads to the condition $\rho_{\cald}>\nu$ being satisfied more often in Step 2 of the algorithm. 

We now investigate the influence of different increasing rules for the subsampling sizes. In the previous experiments, the subsampling size was updated like $\Delta N_k^i=0.02N$.

In order to analyze how the growth strategy of the subsampling size affects the behavior of the method, we consider three different increasing rules. The third one corresponds to the rule used in the previous experiments, while the other two introduce alternative strategies for enlarging the sample size during the iterations.

More precisely, we compare the following three variants of the algorithm, each associated with a different increasing rule for the subsampling sizes for MNIST dataset.

\begin{table}[H]
\begin{center}
\begin{tabular}{|c|c|}
\hline
(MNIST) & Increasing rule \\ \hline
ASMOP\_+1 & $\Delta N_k^i=1$  and $C_2=1$\\ \hline
ASMOP\_2\% & $\Delta N_k^i=0.02N$ and $C_2=1$ \\ \hline
ASMOP\_2\%\_RELAX & $\Delta N_k^i=0.02N$ and $C_2=100$ \\ \hline
\end{tabular}
\caption{\footnotesize{MNIST dataset. Different increasing rules for the subsampling size.}}
\label{tab2}
\end{center}
\end{table}

As in the previous experiments, the algorithms are compared using the criticality measure $\omega(x_k)$ with respect to the number of function evaluations and the execution time. The problem being solved is again \eqref{logregf}, and all the parameters are kept the same as in the previous experiment unless otherwise specified.

\begin{figure}[H]
\centering
\includegraphics[width=5.7cm]{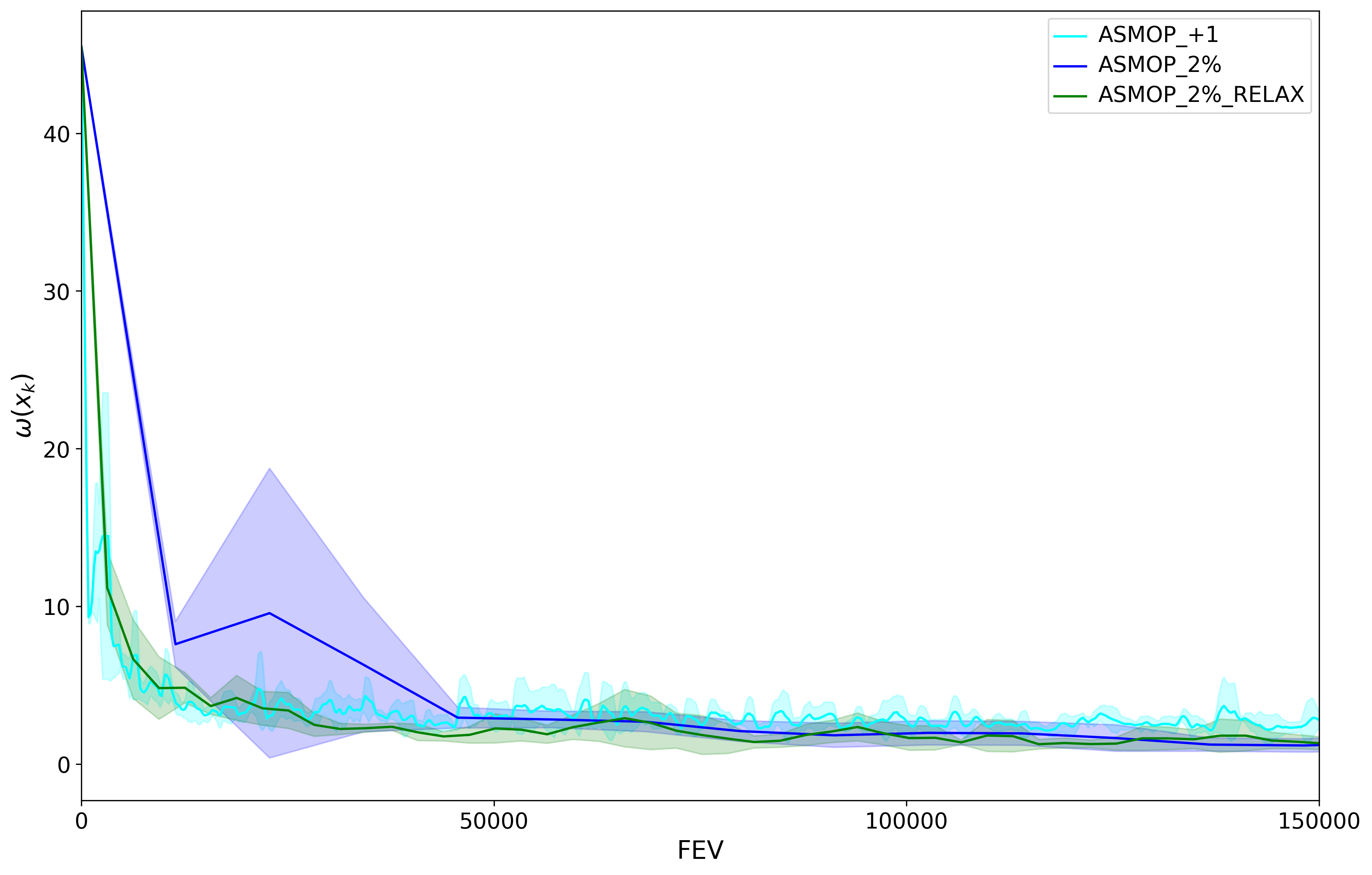}\quad
\includegraphics[width=5.7cm]{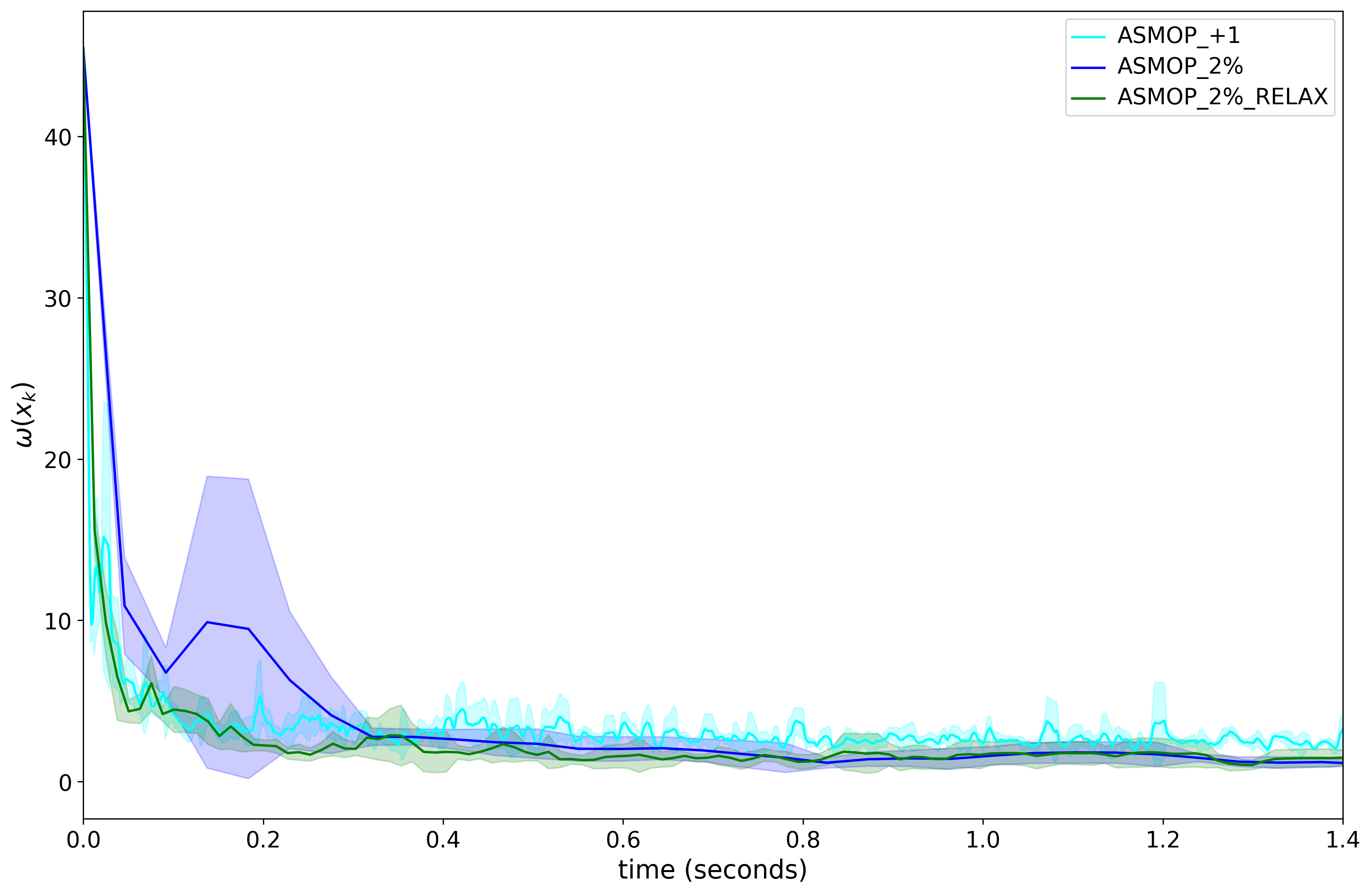}\\
\includegraphics[width=6.27cm]{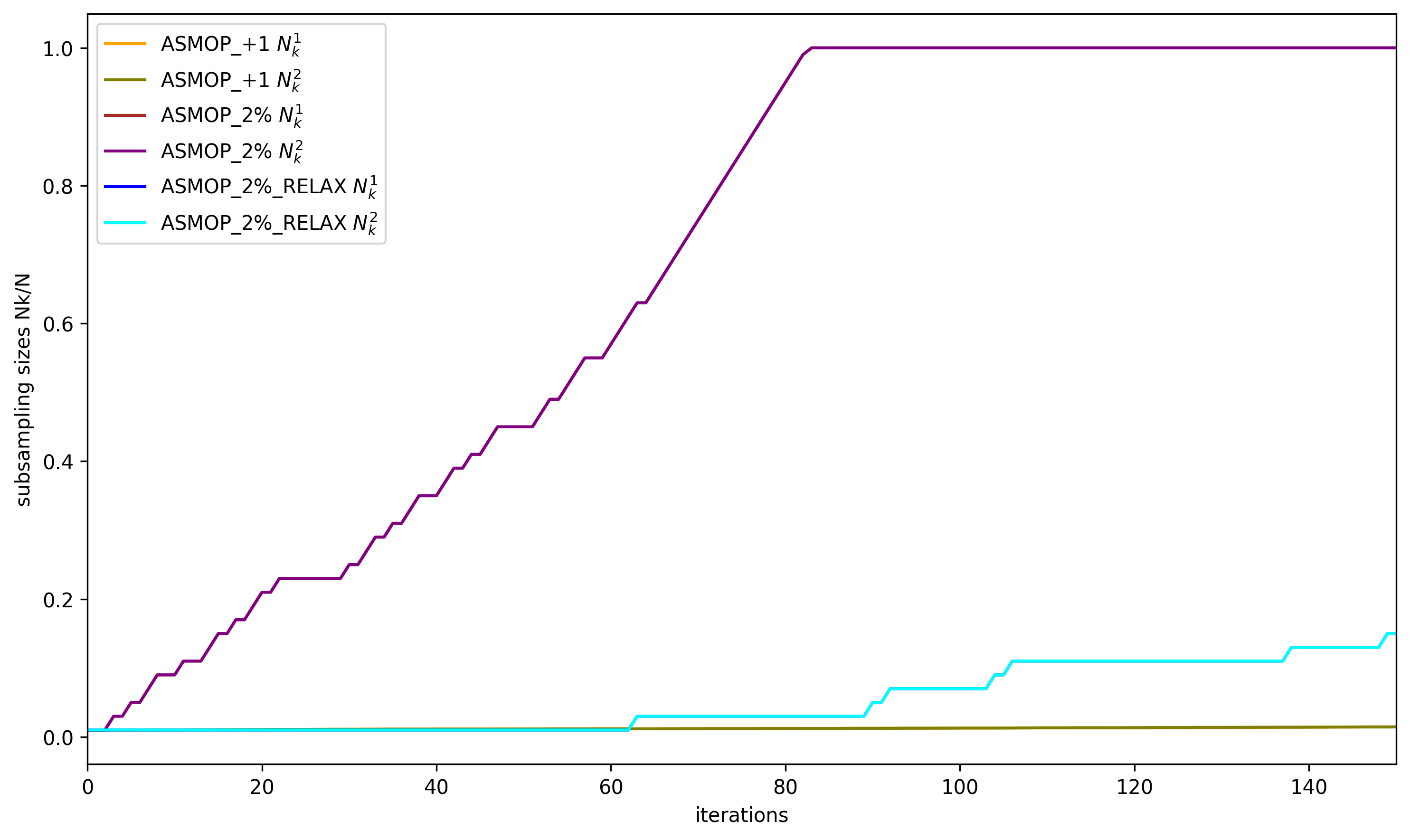}

\caption{{\footnotesize{MNIST dataset, problem \eqref{logregf} different settings Table \ref{tab2}, $N=10^4, n=1024$. Optimality measure against computational cost. Parameters: $x_0=(0.1,0.1,...,0.1), \delta_0=1, \delta_{max}=8, \gamma_1=0.5, \gamma_2=2, \nu=10^{-4}, \eta=0.25,\varepsilon=10^{-4}.$}}}	
\label{mnincrease}
\end{figure}

It can be observed that the different increasing rules lead to significantly different growth patterns of the subsampling sizes. In particular, the rule used in ASMOP\_2\% increases the sample size very aggressively, quickly reaching the full dataset. On the other hand, ASMOP\_+1 enlarges the sample very slowly. 

The strategy adopted in ASMOP\_2\%\_RELAX provides a more balanced behavior, allowing the algorithm to work with smaller subsamples for a longer portion of the iterations while still progressively increasing the accuracy of the function evaluations. As a consequence, ASMOP\_2\%\_RELAX achieves a favorable trade-off between computational cost and convergence speed, which motivated its selection  in the rest of the experiments. The choice of optimal hyper-parameters in an important question and it will be addressed in the future work. 

\section{Conclusion}
Stochastic trust region algorithm for unconstrained multi-objective finite sum problems has been  proposed. 
The method is featured by adaptive sample strategy which, depending on the problem, yields a mini-batch or increasing sample size scheduling.  The sample size is governed by the additional sampling approach, thus the ASMOP method can be viewed as a generalization of method proposed in \cite{NKNKJ} for single-objective problems. The adaptation to multi-objective setup required nontrivial modifications, including the additional sampling criterion. Theoretical analysis also required nontrivial adjustments  combining  additional sampling and  multi-objective analysis. We proved almost sure convergence towards Pareto critical points under assumptions that are standard for multi-objective and additional sampling framework, covering a large class of machine learning problems. Numerical study conducted on different machine learning models, covering both convex and nonconvex functions, showed the potential of the proposed method and its competitiveness with the relevant existing methods. We also presented a short analysis  of different parameters influencing the sample size dynamics which revealed a direction to follow in some future research.

 \section*{Acknowledgments} 

N. Krklec Jerinki\'c and L. Rute\v{s}i\'c were supported by the Science Fund of the Republic of Serbia, GRANT No 7359, Project LASCADO. The work of I. Trombini was supported by “Gruppo Nazionale per il Calcolo Scientifico (GNCS-INdAM)” (Progetti 2026).
\section*{Availability of Data and Materials} The datasets analyzed during the current study are available in links given
in the paper.
\section*{Conflict of interest} The authors have no relevant financial or non-financial interests to disclose.

\end{document}